\documentclass[cntp]{ipart}

\RequirePackage{hyperref}

\startlocaldefs

\renewcommand{\theequation}{\thesection.\arabic{equation}}

\usepackage[OT2,T1]{fontenc}
\DeclareSymbolFont{cyrletters}{OT2}{wncyr}{m}{n}
\DeclareMathSymbol{\Sha}{\mathalpha}{cyrletters}{"58}

\newcommand{\be}{\begin{equation}}
\newcommand{\ee}{\end{equation}}
\newcommand{\bear}{\begin{eqnarray}}
\newcommand{\eear}{\end{eqnarray}}
\newcommand{\ie}{{\it i.e.}}
\newcommand{\GCD}{\rm{gcd}}

\theoremstyle{plain}
	
\newtheorem{theorem}{Theorem}
\newtheorem{lemma}[theorem]{Lemma}
\newtheorem{proposition}[theorem]{Proposition}
\newtheorem{corollary}[theorem]{Corollary}
\theoremstyle{definition} 
\newtheorem*{remark}{Remark}
\theoremstyle{definition} 
\newtheorem{definition}{Definition}

\endlocaldefs

\pubyear{\hspace*{-0.65cm}  \colorbox{white}{\color{white}  \parbox[t]{6cm}{ $\;$ \\ [-3mm]  \color{white} Accepted for publication: March 2022}  } }

\begin{document}

\vspace*{-2.7cm}

\hspace*{-0.72cm} \parbox[t]{16cm}{\small  \color{black} Revised version: March 9, 2022}  \\ \indent
\hspace*{-0.72cm} \parbox[t]{16cm}{\small  \color{blue} Accepted for publication: } 

\vspace*{-0.8cm}

\noindent \makebox[9.5cm][l]{\small \hspace*{-2cm} }{\small Fermilab-Pub-20-577-T}  \\  [-1mm]

\vspace*{1.9cm}

\begin{frontmatter}

\title{Diophantine equations with sum of cubes and cube of sum\protect\thanksref{T1}}
\thankstext{T1}{This work was supported by Fermi Research Alliance, LLC under Contract DE-AC02-07CH11359 with the U.S. Department of Energy.}

\begin{aug}
\author{\fnms{Bogdan A.} \snm{Dobrescu}\ead[label=e1]{bdob@fnal.gov}}
\address{Particle Theory Department, Fermilab, Batavia, IL 60510, USA\\ \printead{e1}}
\and
\author{\fnms{Patrick J.} \snm{Fox}
\ead[label=e2]{pjfox@fnal.gov}}
\address{Particle Theory Department, Fermilab, Batavia, IL 60510, USA\\\printead{e2}}
\end{aug}
\received{\sday{27} \smonth{9} \syear{2021}}

\begin{abstract} 
We solve Diophantine equations  of the type $ a \,  (x^3 \!+ \! y^3 \!+ \! z^3 ) = (x \! + \! y \! + \! z)^3$, 
where $x,y,z$ are integer variables,  and the coefficient $a\neq 0$ is rational. 
We show that there are infinite families of such equations, including those where $a$ is any cube or certain rational fractions, that have nontrivial solutions.
There are also infinite families of equations that do not have any nontrivial solution, including those where $1/a = 1- 24/m$  with restrictions on the integer $m$.
The equations can be represented by elliptic curves unless $a = 9$ or 1, and any elliptic curve of nonzero $j$-invariant and torsion group 
$\mathbb{Z}/3k\mathbb{Z}$ for $k = 2,3,4$, or $\mathbb{Z}/2\mathbb{Z} \times \mathbb{Z}/6\mathbb{Z} $ corresponds to a particular $a$.
We prove that for any $a$ the number of nontrivial solutions is at most 3 or is infinite, and for integer  $a$ it is either 0 or $\infty$. 
For $a = 9$, we find the general solution, which depends on two integer parameters. 
These cubic equations are important in particle physics, because they determine the fermion charges under the $U(1)$ gauge group.
\end{abstract}

\begin{keyword}[class=AMS]
\kwd[Primary ]{11D25}   
\kwd[; secondary ]{11G05,   	
11D45,  	
11D85}      
\end{keyword}

\begin{keyword}
\kwd{Cubic Diophantine equations}
\kwd{Elliptic curves}
\kwd{Primitive solutions}
\kwd{Fibonacci numbers}
\end{keyword}

\tableofcontents

\end{frontmatter}

\section{Introduction} \setcounter{equation}{0}
\label{sec:intro}

The study of elliptic curves provides powerful tools for solving cubic Diophantine equations
\cite{Silverman2009,Cremona1997,Smart1998}.
Nevertheless, it remains challenging to find general solutions to parametric families of 
cubic Diophantine equations  \cite{Hardy, Dickson}. 
Here we study the following family of cubic equations in the integer variables $x,y,z$:
\be
a \left(x^3  + y^3  + z^3  \right) =   (x + y + z )^3  ~~,
\label{eq:general3variables}
\ee
where the coefficient $a$ is a rational number.

These homogeneous Diophantine equations 
have applications in particle physics, related to the allowed charges of spin-1/2 particles under a new $U(1)$ gauge group 
\cite{Appelquist:2002mw, Cui:2017juz, Allanach:2019uuu, Costa:2019zzy}.
Since $a=0$ is a trivial case ($ x = - y - z$), we will focus on $a\neq 0$. 
To characterize the solutions to (\ref{eq:general3variables}), 
we introduce some definitions.

\vspace{0.3cm}  

\begin{definition}
Two integers, $x$ and $y$, form a {\it vectorlike pair} if $x + y = 0$.  
\end{definition}

\vspace{-0.2cm}

\noindent
This terminology is common in quantum field theory, where $x$ and $y$ are fermion charges, 
and the gauge field has vector couplings to the fermions if  $x + y = 0$.  

\begin{definition}  \label{def:primitive}
\ A {\it primitive solution} to Eq.~(\ref{eq:general3variables}) is a solution $\{ x,y,z \}$ with $x,y,z \in \mathbb{Z}$ that satisfies the following conditions: \\
1. $x y z \neq 0$. \\
2.  gcd$(x,y,z ) = 1$.    \\
3.  the set  $ \{ x,y,z  \}  $ does not include any vectorlike pair.    
\end{definition}

\vspace{-0.2cm}

\noindent
Note that for $a \neq 1$  the condition $x y z \neq 0$  implies that  the set  $ \{ x,y,z  \}  $ does not include any vectorlike pair. Thus, the third condition is relevant only for $a=1$. 

\smallskip

\begin{definition} \label{def:general}
A {\it general solution} to the Diophantine equation (\ref{eq:general3variables}) for fixed $a$ is a solution 
which depends on at most 2 integer parameters 
such that any primitive solution can be obtained, after removing gcd$(x,y,z)$ and up to a reordering or an overall sign, 
for some values of the parameters.
\end{definition}

Although we refer to Eq.~(\ref{eq:general3variables}) as ``cube of sum proportional to the sum of cubes'', it is useful to note 
that the change of variables
\bear
&& x = \frac{1}{2} \left(  t - u + v \right)   ~~,
\nonumber \\ [2mm]
&&  y = \frac{1}{2} \left(  t + u - v \right)     ~~,
\label{eq:xyztotuv} \\ [2mm]
&&  z = \frac{1}{2} \left(  -t + u + v \right)      ~~,
\nonumber
\eear
where $t,u,v \in \mathbb{Z}$, transforms Eq.~(\ref{eq:general3variables}) into a cubic equation involving the product of the variables and the 
cube of their sum:
  \be
 \left(  a-1 \right)  \left( t+u+v \right)^3 =   24 a \, t \, u \, v  ~~.
\label{eq:general3tuv}
  \ee
This equation has solutions for $a = 1$ only with $tuv = 0$, corresponding to $\{ x,y,z \}$  solutions to Eq.~(\ref{eq:general3variables})  
which include a vectorlike pair. Thus, there are no primitive solutions for $a = 1$. 

Eq.~(\ref{eq:general3variables}) for $a\neq 0,1$ is also related to a parametric family of  elliptic curves over rational numbers. To see that, let us 
change variables through
\bear
&& x     
=    3 a \left( 3 a - X  \right) \, {\cal R}   ~~~,
\nonumber \\ [2mm]
&& y  =   \left(   \frac{9}{2} \,  a^2   \left( a - 1 \right)  + Y    \right)   {\cal R}   ~~~,
\label{eq:anomtoelliptictrans}
\\  [2mm]
&& z =   \left(  \frac{9}{2} \,  a^2  \left( a - 1 \right)  - Y  \right) \, {\cal R}    ~~~,
\nonumber
\eear
where  $X,Y$ are rational variables, and ${\cal R} \in \mathbb{Q}$ 
is an overall normalization chosen such that $x,y,z \in \mathbb{Z}$. 
Eq.~(\ref{eq:general3variables}) then becomes
\be
Y^2  = X^3 - 27 a^3 X - \frac{27}{4} \, a^4   \left( (a-3)^2  - 12 \rule{0mm}{4mm}  \right)   ~~. 
\label{eq:ellipticcurve}
\ee 
The discriminant of this elliptic curve is 
\be
\Delta= - 3^9 \, a^8  \left( a-1  \right)^3 \,  \left( a-9  \right)   ~~.
\label{eq:disc}
\ee
The only relevant value of $a$ that makes the discriminant  
vanish is $a=9$. It turns out that this is also the most interesting case for simple extensions of the Standard Model of particle physics \cite{Appelquist:2002mw}. 

In  Section \ref{sec:infinite-solutions} we show that any equation of the type (\ref{eq:general3variables}) with $a$ given by the ratio of any perfect cubes, or by certain rational fractions, has nontrivial solutions. 
In Section \ref{sec:nosolution} we  prove that there are no primitive solutions for certain 
 infinite families of rational $a$, and then use properties of the elliptic curves to prove that the number of primitive 
 solutions is at most 3 or is infinite.  
 Section \ref{sec:integera} focuses on integer values of $a$, 
including all the solutions to Eq.~(\ref{eq:general3variables})  in the particular case where 
two variables are equal. This provides a further restriction on the number of primitive solutions, which for integer $a$ can be 0 or $\infty$. 
We also present there a 2-parameter solution  for $a=9$, and prove its generality.
Then, we identify solutions for infinite sequences of  $a$ equal to a perfect square
({\it e.g.,} squared Fibonacci numbers of even index), which are relevant for particle physics.
In Section \ref{sec:conc} we summarize our results.

\section{Infinite families of equations with solutions}     
\setcounter{equation}{0} 
\label{sec:infinite-solutions}

The existence of primitive solutions to Eq.~(\ref{eq:general3variables}) depends on the value of $a$.
In this section we identify a few infinite families of equations with rational values for $a$ that allow primitive solutions.

\subsection{Primitive  solutions when $a$ is a rational fraction}

 \begin{theorem}\label{Theorem1}    
  The equation with integer variables $x,y,z$ 
 \be
p^3 \left(x^3  + y^3  + z^3  \right) = q^3 (x + y + z )^3  ~~
\label{eq:general3rational}
 \ee
 has at least one  primitive solution for any  $p, q \in  \mathbb{Z}$ with $q/p \neq 0, 1, -1/2$. 
 \end{theorem}
\begin{proof}
Let us express the variables as cubic polynomials in $p$ and $q$ as follows:
\bear
&&  x =  \left( p + 2 q  \right)  \left( p^2 + p q + 4 q^2 \right)  ~~, 
\nonumber \\ [2mm]
&&  y =  -3 q \left(p^2  + 2 p q +  3 q^2\right)     ~~,  
 \label{eq:ajk} \\ [2mm]
&&  z =    - p^3 - 3 p^2 q - 6 p q^2 + q^3    ~~.
 \nonumber
\eear
It can be checked that the above expressions represent a solution to Eq.~(\ref{eq:general3rational}) for any 
$p,q$. The necessary conditions $x\neq 0$ and $y\neq 0$   imply $p \neq -2 q$ and $q \neq 0$, while
$z\neq 0$ is satisfied for any rational $p/q$.
To identify the primitive solutions, it remains to impose that no two variables form a vectorlike pair, {\it i.e.},
\bear
&&  x + y  =   p^3 -  q^3  \neq 0 ~~,
\nonumber \\ [2mm]
&&  y + z = -  \left( p + 2q \right)^3   \neq 0 ~~,  
\\ [2mm]
&&  z + x =  9 q^3 \neq 0 ~~,   
 \nonumber 
\eear
so that $p \neq q$.  If  {\rm \GCD}$(x,y,z ) \neq 1$, then $\{ x,y,z \}$ is not a primitive solution. In that case, 
a primitive solution is $\{ x,y,z \}/${\rm \GCD}$(x,y,z) $ with $x,y,z$ given in (\ref{eq:ajk}).
\end{proof}
\begin{remark}
The solution constructed in (\ref{eq:ajk}) is not generically unique. For example, $p = 2$, $q = 1$, ({\it i.e.}, $a = 8$) leads to 
$\{ x,y,z \} =\{40, -33, -31\}$, but the simpler solution $\{ 5,4,3\}$ is not included in  (\ref{eq:ajk}).
\end{remark}

\begin{proposition}\label{Proposition-a}
Equation {\rm (\ref{eq:general3variables})}  has at least one primitive solution iff $\, a$ is of the form  
    \be
a = \frac{(s+2)^3}{6 r^2 + s^3 + 2}     
\label{eq:arsfamily}
  \ee
for any $r, s \in \mathbb{Q}$ satisfying $s \neq 0$,  $|r| \neq 1$ and $r \neq  \pm (s+1) $.
\end{proposition}

\begin{proof}
Put  $r \! = \! r_1/r_2$,  $s \! = \! s_1/s_2$,  with $r_1, r_2, s_1, s_2   \in \mathbb{Z}$,  $r_2  s_2,  \neq 0$ gcd$( r_1,r_2 ) \! = \! $ gcd$( s_1, s_2 ) \! = \! 1$.
The triple $\{ x,y,z \} = \{ s_1 r_2  \, ,   s_2 ( r_2  + r_1 )   \, ,   s_2 ( r_2   -  r_1 )  \}  $
is a solution to  Eq.~(\ref{eq:general3variables}) for $a$ given by (\ref{eq:arsfamily}). 
This is not a vectorlike solution (see Definition \ref{def:primitive}) iff $r \neq  \pm (s+1)$.
The  $xyz \neq 0$ condition is satisfied for  any  $r \neq \pm1$, $s \neq 0$.
Note that gcd$( x,y,z )  =  2^\delta \,$gcd$(s_2, r_2 )$, where $\delta = 1$ if $s_1$ is even and $r_1 r_2$ is odd, and $\delta = 0$ otherwise.  Thus,  
\be
\frac { \{ \, s_1 r_2  \;  ,  \,  s_2 ( r_2  + r_1 )   \;  ,  \,   s_2 ( r_2   -  r_1 )  \,  \}   } { 2^{\delta} \,  {\rm gcd}( s_2, r_2 ) }
\label{eq:arsfamilyFermat}
\ee
is a  primitive  solution.
To prove the converse, for any primitive solution $\{ x,y,z \}$, the value of $a$ extracted from Eq.~(\ref{eq:general3variables}) is identical to 
(\ref{eq:arsfamily}) with 
\be
r = \frac{x - y}{x + y}
\;\;\; , \;\;\; 
s = \frac{2 z}{x + y}  ~~~.
\ee 
Since $xyz (x+y)(y+z) \neq 0$, the values $r =\pm 1$, $s = 0$, and $r = \pm (s+1) $ are not allowed.
Given that $(6 r^2 + s^3 + 2)(x+y)^3 =  x^3 + y^3 + z^3$, 
Fermat's Last Theorem (FLT)  for exponent 3 ensures that the denominator of (\ref{eq:arsfamily})  is nonzero.
\end{proof}

\begin{remark}
For $s = 1$  and  $r = p/q$ with $p, q \in \mathbb{Z}$ coprime,  (\ref{eq:arsfamily}) gives the  infinite family of $a$ values
\be
a = \frac{9 q^2}{ 2 p^2 +  q^2}    ~~,
\label{eq:afrac2}
\ee
which have primitive solutions iff $q\neq 0$ and $\pm p/q \neq 1,2$: $\{ x,y,z \} = \{ p+q \, , \, q  \, , \,  - p + q  \}$.
Within this family, there are only three integer values of $a$ that allow primitive solutions:
$p \! = \! 1$, $q  \! = \! 2$ gives $a=6$ and  the solution $\{ 3, 2,  1 \}$;
$p \! = \! 1$, $q \! = \! 4$ gives $a=8$ and $\{ 5, 4, 3 \}$;
$p \! = \! 0$, $q \! = \! 1$ gives $a=9$ and  $\{ 1,1,1 \}$.

An infinite family with the coefficient $a$ given by the reciprocal of an integer is obtained 
from Proposition \ref{Proposition-a} by setting $r = p/q$, $s = -2 + 1/q$, with $q  \notin \{ 0, \pm p, 1\pm p\}$.
The ensuing family of equations, spanned by the two integers $p,q$, is 
\be
x^3  + y^3  + z^3 =  \left[ 6 q\left( p^2  - (q - 1)^2 \right) + 1  \rule{0mm}{3.9mm} \right] (x + y + z )^3 ~~,
\ee 
and has the primitive solution $\{ 1 - 2 q , \, p + q    , \, - p + q    \} $.

Many other simple families may be obtained from   Proposition \ref{Proposition-a}. 
For example, 
$s = - 1$  gives $1/a = 6 r^2 + 1$, $s = -3 $ gives $1/a = 25 - 6 r^2$,  and $s = 2$ gives $a = 32/ (3 r^2 + 5)$.
Some infinite families of integer $a$ with primitive solutions are derived in Section~\ref{sec:integera-inf}. 
Despite the versatility of (\ref{eq:arsfamily}), the methods employed in this subsection are not helpful for proving Theorem \ref{Theorem1}
because it is nontrivial to find the solutions (\ref{eq:ajk}). 
\end{remark}

\subsection{General solution for two variables} 
\label{sec:n2variables}

In this subsection we seek all the solutions to Eq.~(\ref{eq:general3variables}) that have one of the variables equal to zero. 
It is sufficient to consider the Diophantine equation in two variables:
  \be
a \left(x^3  + y^3   \right) =   (x + y )^3  ~~,
\label{eq:aan2}
  \ee
where  $a \in \mathbb{Q} $.   
Trivial (vectorlike) solutions with $y = -x$ exist for any $a$. If $xy = 0$, then solutions with 
a single nonzero variable exist only for $a =1$. Leaving aside these trivial cases, the following Proposition provides the full solution to 
the equation with ``cube of sum proportional to the sum of cubes" 
in two variables. As we seek nontrivial solutions, we remove 
an $x+y$ factor, so that  Eq.~(\ref{eq:aan2}) is equivalent to (\ref{eq:an2xy}) given below.  

\begin{proposition}\label{Proposition1}
The equation with integer variables $x, y$
  \be
(a-1) (x^2 + y^2) = (a+2) \,  x\, y      
~~,
\label{eq:an2xy}
  \ee
where $a \in \mathbb{Q}$, has solutions with $xy(x+y )\neq 0$  iff $a$ is of the form
\be
a = \frac{4}{1+ 3 \, \alpha^2 }   ~~,
\label{eq:an2rational}
\ee
where $\alpha$ is any rational number except $\pm 1$. 
\end{proposition}

\begin{proof}
If $a$ is given by  (\ref{eq:an2rational}), then it is straightforward to check that any integers $x$, $y \neq 0$ that satisfy
\be
\frac{x}{y} = \frac{\alpha\pm1}{1\mp\alpha}      \;\;\;\;   {\rm for } \;\,\;  \forall  \, \alpha\neq 0,\pm 1 \; , \;\;\;\;    \mathrm{or}  \;\;\;\;  x=y \;\;\;  \mathrm{for}  \;\;\;\alpha=0
\label{eq:xyr}
\ee 
provide a solution to (\ref{eq:an2xy}). 

To prove the reverse, we note that if $\{ x, y\}$ is a solution to  (\ref{eq:an2xy}) with $x+y \neq 0$,
then the equation can be written as
\be
\frac{4}{a } = 1+3 \left( \frac{x-y}{x+y} \right)^{\! 2} ~~,
\ee
so that $a$ is given by  (\ref{eq:an2rational}) with 
\be
\alpha= \pm \frac{x-y}{x+y}  ~~.
\label{eq:rxy}
\ee
\end{proof}

The first part of Proposition~\ref{Proposition1} proves in particular that there is an infinite family of  
rational values for $a$, given in (\ref{eq:an2rational}), for which
(\ref{eq:general3variables}) has solutions with $xy(x+y )\neq 0$ and $z=0$.
The following corollary shows that the situation is different when $a \in \mathbb{Z}$.

\begin{corollary}\label{Cor1}
Up to a reordering or an overall integer rescaling, there are only two solutions to
Eq.~(\ref{eq:aan2}) with $xy(x+y )\neq 0$  when the coefficient $a$ is an integer: $x = 2$, $y = 1$ for $a = 3$, and $x = y = 1$  for $a = 4$.   
\end{corollary}
\begin{proof}
Eq.~(\ref{eq:an2rational}) implies $0 < a \leq 4$. From Eq.~(\ref{eq:an2xy}) follows that for $a = 1$ there are no solutions with $xy(x+y )\neq 0$.
Also, there are no solutions for $a=2$ because in that case there is no rational value of $\alpha$ that satisfies  (\ref{eq:an2rational}). 
For $a = 3$, $x/y = 2$ or 1/2, while $a = 4$ gives $x=y$. 
\renewcommand{\qedsymbol}{}
\end{proof}

\begin{corollary}\label{Cor2}
The general solution $\{ x, y\}$  to Eq.~(\ref{eq:aan2})  with  $xy(x+y )\neq 0$, up to an overall rescaling by an 
integer and an $x\leftrightarrow y$ interchange, is $\{ x,y \} = \{ p+q \, , \, p - q\}$,
where $p,q \in \mathbb{Z}$ such that $\alpha = q/p \neq \pm 1$ in (\ref{eq:an2rational}).  
\end{corollary}
\begin{proof}
Solving (\ref{eq:rxy}) for $x/y$ gives  (\ref{eq:xyr}) or the same with $x$ and $y$ interchanged.
Replacing $\alpha = q/p$ in (\ref{eq:rxy})  gives $x =  p + q$,  $y =  p - q  $  up to an overall normalization.
\renewcommand{\qedsymbol}{}
\end{proof}
 
Some numerical examples with noninteger $a$, are $p = 2$, $q = 3$, 
which gives  $a=12/7$ and the solution $\{ 5, 1\}$, and  $p = 1$, $q = 2$,  which gives  $a=16/7$ and the solution $\{ 3, 1\}$.

\section{The number of primitive solutions}
\setcounter{equation}{0}
\label{sec:nosolution}

For certain values of $a$ it is possible to prove that there exist no primitive solutions to Eq.~(\ref{eq:general3variables}).  
We invoke the results of E.~Dofs \cite{Dofs1, Dofs2} to show that there are infinite families of rational $a$ (including certain integer values)  for which there are no solutions. We then use of elliptic curves to determine the number of primitive solutions for fixed $a$.

\subsection{Infinite families of equations with no primitive solution}
\label{sec:a-2425}

Using the change of variables presented in  (\ref{eq:xyztotuv}), 
 the ``cube of sum proportional to the sum of cubes" equation can be related to a homogenous cubic equation, 
shown in (\ref{eq:general3tuv}), 
 where the cube of a sum of 3 integers is proportional to their product. 
This allows us to invoke existing results to show that there is a  set of rational $a$ for which there are no solutions,
including  the integer values $a=-23, -5, -3, -2, 4, 7, 25$. 

\begin{theorem}\label{theoremnosolnp}
There are no primitive solutions to Eq.~(\ref{eq:general3variables}) if 
\be
a = \frac{p^{n}}{p^{n}-24}   ~
\label{eq:apn}
\ee 
for $a \neq -1/11$, $p\equiv 2\!\! \mod \! 3$ prime, $n \ge 1$ integer 
with $3\nmid n$, and all prime factors of $p^{n}-27$ congruent to $2\!\! \mod 3$.
\end{theorem}
\begin{proof}
For $a \in \mathbb{Q}$ satisfying (\ref{eq:apn}), 
use the change of variables (\ref{eq:xyztotuv}) to rewrite  (\ref{eq:general3variables}) in the form
\be
 \left( t+u+v \right)^3 =   p^n   
 \, t \, u \, v  ~~.
 \label{eq:tuvp1}
\ee
The variables $t,u,v$ are pairwise coprime, which implies that any solution must take the form $t=t_0^3$, $u=u_0^3$, $v=p^{k}v_0^3$, with 
$k=1$ or 2, $n=3n'-k$, 
$p \nmid t_0,u_0$, and the integer variables $t_0,u_0,v_0$ are pairwise coprime.  Thus, (\ref{eq:tuvp1}) becomes
\be
t_0^3 + u_0^3 + p^{k} v_0^3 = p^{n'} t_0 u_0 v_0~.
\label{eq:Dofstuv}
\ee
We now use a theorem of Dofs \cite{Dofs1}, which investigates solutions to equations of the form 
$p^{\omega_1} t_0^3 + p^{\omega_2} u_0^3 + p^{\omega_3} v_0^3 = d\,  t_0 u_0 v_0$.  
Dofs defines a quantity $F=d^3 - 27p^{\omega_1+\omega_2+\omega_3}$,   
and determines the solvability of the above equation when all prime factors of $F$ are $2 \!\! \mod \! 3$
and certain conditions on $F$, $p$, and $d$ are satisfied.  

For Eq.~(\ref{eq:Dofstuv}) these conditions determine that if $k=1$ (case $N_{01}$ in Dofs' proof) there are no nontrivial solutions  
unless $p=2$ and $n'=2$ (\ie\ $a=4$). In that case the only solution is $t_0 = u_0 = v_0=1$, corresponding to  $\{x,y,z\}=\{1,0,1\}$, which  
is not a primitive solution to (\ref{eq:general3variables}).  
If $k=2$ (case $N_{02}$ in Dofs' proof) there are no nontrivial solutions unless $p=2$ and $n'=1$ (\ie\ $a=-1/11$) when the only solution is $\{t_0,u_0,v_0\}=\{1,1,-1\}$, which corresponds to the primitive solution $\{x,y,z\}= \{-2,3,-2\}$.
\end{proof}

\begin{remark}
The values $a=-23, -2, 4, 25$ are the only integer ones that yield
$24a/(a-1)$ as a non-cubic power of a single prime: $23$, $2^4$, $2^5$, $5^2$, respectively.  Thus, $a=-23$, which corresponds to $n=n'=1$, $k=2$, and $a=-2$ ($n'=k=2$) can be seen to have no solution. For $a=4$ ($n'=2, k=1, p=2$) there is a solitary solution to (\ref{eq:Dofstuv}) which 
is not a primitive solution to (\ref{eq:general3variables}).   
The case of $a=25$ corresponds to $p=5, k=1, n'=1$ and has no solution (see also \cite{ward}).  
\end{remark}

\smallskip

\begin{theorem}\label{theromDofs2}
There are no primitive solutions to Eq.~(\ref{eq:general3variables}) for 
\be
a = \frac{p \, q^{2}}{p \, q^{2}-24}~,
\label{eq:apq}
\ee
with $(p,q) \in \{ (2,3), (5,2), (7,2) \} \cup \{ (2,Q)|\,\, Q \equiv 1\!\! \mod \! 3$ prime, $2Q^2-27$ prime, and $4$ is a cubic non-residue {\rm mod}$\,  Q \, \}$.
\end{theorem}
\begin{proof}
Consider the rational coefficient $a$ of the more general form
\be
a = \frac{p^{n}q^{m}}{p^{n}q^{m}-24}~,
\label{eq:atwoprimes}
\ee
where $p\ne q$, $n=3n'-k_p$ and $m=3m'-k_q$ with $k_p,k_q=1$ or $2$.  
Changing variables from $\{ x,y,z \} $ to  $\{ t,u,v \} $ as in (\ref{eq:xyztotuv}), Eq.~(\ref{eq:general3variables}) becomes
\be
 \left( \frac{ t+u+v }{ p^{n'}   q^{m'} } \right)^{\! 3} =  \frac{t \, u \, v}{ p^{ k_p } \,   q^{ k_q }  }   ~~.
 \label{eq:tuvp2}
\ee
The fact that $t,u,v$ are pairwise coprime  implies that at least one variable is a perfect cube.
Without loss of generality, we take $t = t_0^3$. For the other variables, up to a reordering, there are two cases:
\be
 u =  u_0^3  \;\; , \,\,\, v = p^{k_p} q^{k_q} v_0^3 
 \;\; \,\,\,\,   \;\;    {\rm or}    \;\;  \,\,\,  \;\; 
 u = p^{k_p} u_0^3  \;\; , \,\,\, v = q^{k_q} v_0^3  ~~~.
\ee
The first case is similar to the one encountered in the derivation of Eq.~(\ref{eq:Dofstuv}), and implies
 \be
t_0^3 + u_0^3 + p^{k_p}  q^{k_q}v_0^3 = p^{n'}q^{m'} t_0 u_0 v_0  ~~~.
\label{eq:t0u0v0}
\ee
In the second case,  defining new variables
\bear
T &=&  t^2  u + u^2 v + v^2 t - 3 \, t \, u \, v  \nonumber  \\ [2mm]
U &=&  t \, u^2 + u  \, v^2 + v \, t^2 - 3 \, t \, u \, v   \\ [2mm]
V &=& \frac{1}{p^{n'} q^{m'}} \left(t^3+u^3+v^3- 3 \, t \, u \, v \right)     \nonumber 
\eear
transforms Eq.~(\ref{eq:tuvp1}) into a cubic Diophantine equation of the same form as (\ref{eq:t0u0v0}),
\be
T^3 + U^3 + p^{k_p}q^{k_q} V^3 = p^{n'} q^{m'} T U V  ~~~.
\label{eq:DOFsPQR}
\ee  
Thus, if this equation has no solution, neither will  (\ref{eq:general3variables})  with $a$ given by (\ref{eq:atwoprimes}).
Dofs \cite{Dofs2}  has carried out an investigation of the conditions under which (\ref{eq:DOFsPQR}) is unsolvable for $k_q=1$ and $k_p\le 3$.  Restricting further to $n=1, m=2$ leads to the result of use here.
We refer the reader to \cite{Dofs2} for details but point out that $(p,q)=(2,3)$ satisfies the unsolvability conditions of case $(ii,c)$ of Dofs \cite{Dofs2}, while all other values for $(p,q)$ listed after (\ref{eq:apq}) satisfy case $(i)$ of Dofs  and so can be determined to have no solution.
\end{proof}

The first three values for $(p,q)$ targeted by Theorem  \ref{theromDofs2}  lead  to integer $a$: 
$-3, -5$, and $7$, respectively.  The remaining values form a  family of rational $a$ for which there are no solutions:
\be
a = \left( 1 - \frac{12}{Q^2} \right)^{-1} ~~~,
\ee
where $Q$ satisfies the conditions of Theorem  \ref{theromDofs2}. 
The following Proposition addresses the structure of the prime numbers $Q$ that define this family.

\smallskip

\begin{proposition}\label{Proposition2}
For any $Q \equiv 1\, $mod$\; 3$ prime, a necessary condition for $4$ to be a cubic non-residue$\!\mod Q$ is that 
\be
Q = \frac{1}{4} \left( L^2 + 27 M^2 \right)  ~~
\label{eq:QLM}
\ee
with $L$ and $M$ odd. 
\end{proposition}
\begin{proof}
\ By definition,  $4$ is a cubic non-residue  mod$\, Q$ if there exists no integer $s$ with $s^3 \equiv 4\,$mod$\, Q$. 
A necessary condition is that  $2$ is a cubic non-residue mod$\,Q$. 
It can be shown \cite{IrelandRosen} that any $Q \equiv 1\, $mod 3 prime can be written
as (\ref{eq:QLM}) with $L,M >0$ uniquely determined integers. 
Furthermore, $2$ is a cubic residue mod$\, Q$ iff $L$ and $M$ are even  \cite{Lemmermeyer}. 
Thus, $2$ is a  cubic non-residue mod$\, Q$ iff both $L$ and $M$ are odd.
\end{proof}

The first four values of $Q$ that satisfy the conditions of Theorem \ref{theromDofs2} are 
$Q = 7, 13, 37, 67$, and the corresponding values of $(L,M)$ are (1,1), (5,1), (11,1), (5,3). 
The ensuing values of $a$, for which there are no primitive solutions to Eq.~(\ref{eq:general3variables}),
are  49/37, 169/157, 1369/1357, 4489/4477.

\smallskip\smallskip

\subsection{Testing for  primitive solutions with elliptic curves}
\label{sec:elliptic}

Using the transformation of   (\ref{eq:anomtoelliptictrans}),
the ``cube of sum proportional to the sum of cubes'' equation can be converted to an elliptic curve ($E$), 
given in (\ref{eq:ellipticcurve}), allowing one to bring to bear all of the powerful machinery associated with their study \cite{Silverman2009,Cremona1997,Smart1998}.  The Mordell-Weil theorem \cite{MordellWeil}  states that the set of rational points on an elliptic curve, $E(\mathbb{Q})$, form a finitely generated Abelian group and as such can be separated into sets of points of finite order, the torsion group(s), and the set of points of infinite order.  The number of infinite order points is the rank of the Mordell-Weil group of $E(\mathbb{Q})$. 
Thus, if the rank is $0$, then the number of rational points on $E$ is finite, and the points are generated by the members of the torsion group. 

The rank and the generators of the torsion group can be efficiently calculated, for any given value of $a$, by computer algebra systems such as \texttt{PARI/GP} \cite{PARI2}.  
Some of the algorithms rely on the Birch and Swinnerton-Dyer conjecture \cite{BSD}, which relates the rank of $E(\mathbb{Q})$ to the order of the zero of a certain $L$-function, $L(E,s)$, at $s=1$, the so-called analytic rank. The conjecture has been proved  \cite{Kolyvagin} for curves whose rank is $0$ or $1$,
and thus this procedure is adequate for our purpose of determining the number of solutions to Eq.~(\ref{eq:general3variables}) for given $a$.

The elliptic curve (\ref{eq:ellipticcurve}) is in the Weierstrass reduced form,
\be
Y^2  = X^3 + A X + B  ~~, 
\label{eq:Weierstrass-reduced}
\ee 
with $A = - 27 a^3$, $B = - (27/4) a^4  \left( a^2 - 6 a - 3  \right) $. The $j$-invariant 
of this curve is 
\be
j_a = -  \frac{(24A)^3}{\Delta} = - \frac{2^{12} \, 3^3 \,   a }{(a-1)^3\, (a-9)}  ~~,
 \label{eq:jinvariant}
\ee
where we used the discriminant $\Delta = - 16  \left(  4 A^3 + 27 B^2 \right) $ computed in (\ref{eq:disc}).
Birational transformations,  $X \to r_b^2 X$, $Y\to r_b^3Y$ with  $r_b \in \mathbb{Q} $, 
and more complicated isogenies do not change the $j$-invariant. 

For the curve (\ref{eq:Weierstrass-reduced}) with $\Delta \neq 0$ and integer $A,B$,  
the Nagell-Lutz theorem \cite{NagellLutz}  ensures that any rational point that belongs to the torsion group has integer coordinates. 
Thus, when the  $A,B$ coefficients are noninteger, it is useful to perform a birational transformation that gives integers coefficients.
For any $a = a_L/a_R$, with  $a_L, a_R \in \mathbb{Z} $, the elliptic curve (\ref{eq:ellipticcurve})
can be transformed into an elliptic curve  with integer coefficients:
\be
Y_0^2  = X_0^3 - 432 \, a_L^3 a_R  X_0 - 432 \, a_L^4   \left( a_L^2 - 6 a_L a_R  - 3 a_R^2 \rule{0mm}{4mm}  \right)   ~~,
\label{eq:ellipticcurveWeier}
\ee 
where the new variables are $X_0 = (2 a_R)^2 X$, $Y_0 = (2 a_R)^3 \, Y$.
The corresponding solution to  (\ref{eq:general3variables}), up to an overall normalization similar to that shown in 
(\ref{eq:anomtoelliptictrans}), is given by 
\bear
x & \!=\! & 6 a_L \left( 12 a_L a_R - X_0 \right)    ~~~,
\nonumber \\ [-2mm]
\label{eq:anomtoelliptictrans0-old}
 \\ [-2mm]
y , z & \!=\! &  36 \,  a_L^2  \left( a_L - a_R \right)  \pm Y_0      ~~~.
\nonumber
\eear

Note that    
the points $\pm P_V$ on the elliptic curve (\ref{eq:ellipticcurveWeier})  of integer 
coordinates  $12a_L^2  \left( 1  ,   \pm 3 (a_L  -  a_R)   \rule{0mm}{3.5mm}   \right)  $ 
correspond to the vectorlike solutions $x+y = z = 0$ and $x+z = y = 0$.
Furthermore, $d^2 Y_0/d X_0^2$ vanishes at  $\pm P_V$, so these two inflection points and the point at infinity (labelled $\cal O_\infty$) form a $\mathbb{Z}_3$ subgroup of the torsion group. 
This implies that the torsion group, which in general can be any of the 15 discrete groups listed in \cite{Mazur},  is constrained in our case to be one of only 5 groups: 
$\mathbb{Z}_{3n}$ with $1 \leq n \leq 4$, and 
  $\mathbb{Z}_2 \times \mathbb{Z}_6$  (we use the short-hand notation  $\mathbb{Z}_{m} = \mathbb{Z}/m\mathbb{Z}$).
In Section \ref{sec:integera} we will show that this list is significantly shorter 
when $a\in \mathbb{Z}$. 
Examples of elliptic curves of various torsion groups and ranks 
can be found at \cite{LMFDB}. These can be transformed into the Weierstrass reduced form and
compared with (\ref{eq:ellipticcurveWeier}) up to a birational transformation,
checking if they are consistent with integer values for $a_L$ and $a_R$.
We next prove some results about  the number and properties of primitive solutions to (\ref{eq:general3variables}).

\begin{proposition}\label{Proposition6sol}
For any set of non-vectorlike integer solutions to Eq.~(\ref{eq:general3variables}) with $a \in \mathbb{Q}$,
$a \neq 0,1,9$,
obtained by a reordering or an overall rescaling of $\{ x,y,z\}$, there exist 6 rational points on 
the elliptic curve (\ref{eq:ellipticcurve}) when $x,y,z$  do not include two equal variables, or 3 rational points when two variables are equal.  
\end{proposition}
\begin{proof}
If  $\{ x,y,z\}$  is a non-vectorlike solution to Eq.~(\ref{eq:general3variables}) with $a\neq 0,1$, then from  (\ref{eq:anomtoelliptictrans}) follows that 
$(X_1 , Y_1)$ are two  rational points on the elliptic curve  (\ref{eq:ellipticcurve}) with 
\be
(X_1 , \pm Y_1) = \left(   3 a -\frac{x}{3 a  {\cal R}_1 }  \;\;   , \;   \pm \frac{y-z}{2   {\cal R}_1  }  \right)
 \;\;   , \, \;  \;
{\cal R}_1  = \frac{y + z}{ 9 a^2 (a - 1) }   ~~.
\label{eq:X1Y1R1}
\ee
Given that $\{ y,z, x\}$ and  $\{ z, x, y \}$ are also solutions, there are two more pairs of rational points, $(X_i , \pm Y_i)$ with $i = 2,3$,
whose coordinates are given by the appropriate permutations of $x,y,z$ in (\ref{eq:X1Y1R1}).
Thus, generically, for all the solutions obtained by a reordering or an overall rescaling of $\{ x,y,z\}$, there are
6 rational points whose coordinates are given by (\ref{eq:X1Y1R1}) and its permutations.     
It is straightforward to check that the 6 rational points are distinct when no two of the $x,y,z$ variables are equal.

If two of the $x,y,z$ variables are equal, then two of the pairs of rational points are identical,
and the third pair has $Y_i = 0$, so only 3 rational points are distinct.
The case with $x=y=z$ is not relevant here, as it occurs only for the singular curve with $a=9$.
\end{proof}

\begin{theorem}\label{TheoremZ9}
The torsion group of the elliptic curve (\ref{eq:ellipticcurve}) is $\mathbb{Z}_9$  iff 
the rational coefficient $a$ in Eq.~(\ref{eq:general3variables}) satisfies
\be
a^{-1} = 1 -  3 \left(\frac{ 2 f ( f - 1) }{f^3 - 3 f^2 + 1} \right)^{\! 3}   ~~,   
\label{eq:Z9a}
\ee
where $f \in \mathbb{Q}$, $f \neq 0, 1$.  
\end{theorem}
\begin{proof}
Following Kubert's parametrization \cite{Kubert}, the elliptic curve in rational variables $X_K, Y_K$
can be written as
\be
Y_K^2 + (1 - c_K)Y_K X_K -b_K Y_K = X_K^2 \left( X_K - b_K \right)    ~~,
\label{KubertCurve}
\ee
where $b_K, c_K \in \mathbb{Q}$, $b_K \neq 0$. 
The point $P_0 = (0,0)$ is of maximal finite order. Following the method of Reichert \cite{Reichert},
we repeatedly use the group addition on $P_0$: 
$2P_0 = (b_K \, , \,  b_K c_K )$,  $3 P_0 =  (c_K \, , \,  b_K - c_K )$. For $c_K \neq 0$ and  $b_K \neq c_K$ we also obtain
\bear
&& 4  P_0 =  \frac{b_K}{c_K^2}   \left(  b_K - c_K \; , \;   b_K (c_K + 1) -  \frac{b_K^2}{c_K}  \right)  ~~,
\nonumber \\ [-2mm]
\\ [-2mm]
&& 5 P_0 = \frac{b_K \, c_K^2}{ (b_K - c_K)^2 }   \left(  c_K + 1 -  \frac{b_K}{c_K}  \; \, , \;  b_K -  \frac{ c_K^3}{b_K - c_K}  \right)  ~~.
\nonumber 
\eear
The above signs in $Y_K(4P_0)$ differ from those in \cite{Reichert}, and agree with \cite{Atkin}.

The torsion group of the elliptic curve is $\mathbb{Z}_9$  iff $X_K(5P_0) = X_K(4P_0)$, which requires
\be
c_K^2 + c_K - b_K = \left( \frac{b_K}{c_K} - 1  \right)^{\! 3}  ~~~.
\ee 
Imposing that the left-hand term is a cube, chosen for convenience to be $(2 f - 1)^3$ with $f \in  \mathbb{Q}$, gives
\be
b_K = c_K \left( f^2 - f + 1 \right)    \;\;\;    ,    \;\;\;     c_K =  f^2 \left( f - 1 \right)    ~~.
\label{eq:f-restriction}
\ee
This agrees with \cite{Kubert}, and differs slightly from  \cite{Battista}. The conditions  $b_K c_K \neq 0$ and  $b_K \neq c_K$
are satisfied for any $f \neq 0,1$.
 
The elliptic curve (\ref{KubertCurve})   can be 
transformed  to the reduced Weierstrass form by putting
\bear
&& X_K=  r_b^2 X + \frac{b_K - c_\star^2}{3}   ~~,
\nonumber \\ [-2mm]
\label{Kubert-transf}
 \\ [-2mm]
&& Y_K = r_b^3 \, Y - c_\star r_b^2 X  -  \frac{c_\star}{3} \left( b_K - c^2_\star \right) + \frac{b_K}{2}  ~~,
\nonumber 
\eear
where $c_\star= (1 - c_K)/2$, and $r_b$ is the parameter of a birational transformation.
The Weierstrass coefficients $A_K$ and $B_K$ are polynomials in $b_K,c_K$, divided by $r_b^4$ and 
$r_b^6$, respectively. 
The requirement that the discriminant of (\ref{KubertCurve}) is nonzero gives the condition 
\be
D_K = 2 b_K^2 
- b_K \left( 4 c_\star^2 - 9 c_\star + 27/8 \right) + 2 c_\star^4 -  c_\star^3  \neq 0 ~~.
\label{eq:DK}
\ee
The $j$-invariant of (\ref{KubertCurve}) and also of the transformed curve is 
\be
j_K = \frac{ 2^9}{D_K}  \left(b_K -2 c_\star^2 +3 c_\star  + c_\star^4/ b_K \right)^3  ~~.
\label{eq:jK}
\ee

In the case of the $\mathbb{Z}_9$  torsion group, the restriction (\ref{eq:f-restriction}) 
ensures that (\ref{eq:DK}) is satisfied, and it implies that 
$A_K$, $B_K$ become polynomials of order 12, respectively 18, in $f$, labelled by $A_f$ and $B_f$.
A necessary condition for our elliptic curve  (\ref{eq:ellipticcurveWeier}) to have 
the $\mathbb{Z}_9$ torsion group is that its points of order 3,
given by $(X,Y) = 3a^2( 1, \pm 3 (a -1)/2  )$,   
are identified with the  points $3 P_0$ or $-3 P_0$ 
on $Y^2 = X^3 + A_f X + B_f$, which have coordinates
\be
\hspace*{3mm}
X(3 P_0)    = \frac{1}{12\, r_b^2}   \left( f^3 - 3 f^2 + 1 \right)^2    \,\;\; , \,\;\;    Y(\pm3 P_0) = \mp \frac{  f^3}{2 \, r_b^3} \left( f - 1 \right)^3   ~.
\ee
This identification implies that $a$ must satisfy (\ref{eq:Z9a}), or the analogous relation with 
a sign flip in front of the parenthesis. 
The $j$-invariant obtained by inserting the restriction (\ref{eq:f-restriction}) in (\ref{eq:jK}) can then be compared 
with (\ref{eq:jinvariant}), and the $j_K = j_a$ identity is obtained only for the sign assignment in (\ref{eq:Z9a}).
Furthermore, no rational value for $f$ allows the equality of the $j$-invariants for the other sign assignment.
Thus, (\ref{eq:Z9a}) is also a sufficient condition for the $\mathbb{Z}_9$ torsion group. 
\end{proof}


\begin{proposition}\label{PropositionTorsiona}
Any elliptic curve with torsion group $\mathbb{Z}_{6}$ and nonzero $j$-invariant, 
or with torsion group $\mathbb{Z}_{9}$, $\mathbb{Z}_{12}$, 
or  $\mathbb{Z}_2 \times \mathbb{Z}_6$,
is equivalent to Eq.~(\ref{eq:general3variables}) for a particular  $a \in  \mathbb{Q}$.
\end{proposition}
\begin{proof} \ 
Equating the $j$-invariants $j_K$, given in (\ref{eq:jK}), and $j_a$, given in (\ref{eq:jinvariant}),  leads to a quartic equation for $a$ in terms of $b_K$, $c_K$.  When the torsion group is $\mathbb{Z}_{6}$, $\mathbb{Z}_{9}$, $\mathbb{Z}_{12}$ or  $\mathbb{Z}_2 \times \mathbb{Z}_6$ the parametrization for $b_K$, $c_K$ given by Kubert \cite{Kubert} leads to a rational root of the quartic.  The general parametrization for an elliptic curve with torsion group $\mathbb{Z}_{6}$ or $\mathbb{Z}_2 \times \mathbb{Z}_6$ is given by (\ref{KubertCurve}) with $b_K=c_K+c_K^2$ for $c_K\ne 0,-1,-1/9$, and the rational solution for $a$ satisfies
\be
a^{-1}= 1-24 \, \frac{c_K^2 (c_K+1) }{(3c_K+1)^3}~~~.
\label{eq:jajKZ6}
\ee
The $j_K = 0$ curve with $\mathbb{Z}_{6}$ torsion has $c_K = -1/3$, and cannot be obtained from (\ref{eq:ellipticcurve})
for any $a$. Any Kubert curve (\ref{KubertCurve}) with  $\mathbb{Z}_{6}$ torsion and $j_K \neq 0$  can be transformed into our elliptic curve (\ref{eq:ellipticcurve})  using  (\ref{Kubert-transf}) with 
\be
r_b = \frac{3 c_K \left(3 + c_K + c_K^2  \right)   + 1}{6 \, \left(3 c_K + 1\right)^2}    ~~. 
\label{eq:rbKZ6}
\ee
For torsion group $\mathbb{Z}_2\times\mathbb{Z}_6$, the same forms for $a$ and $r_b$ as in (\ref{eq:jajKZ6}) and  (\ref{eq:rbKZ6})  are valid with the restriction 
$c_K=2(5-\alpha)/(\alpha^2-9)$, where $\alpha \in \mathbb{Q}$ and $\alpha \ne 1, \pm 3, 5, 9$.

For torsion group $\mathbb{Z}_{12}$ the general parameterization \cite{Kubert} has $b_K$, $c_K$ as rational fractions of a parameter $\tau \in \mathbb{Q}$, $\tau\ne 0,1,1/2$, and the rational solution of the quartic satisfies
\be
a^{-1}=1-24 \, \frac{ \sigma^4  (4\sigma+1)(2\sigma+1) }{ \left( 6\sigma(\sigma+1)+1 \rule{0mm}{3.7mm}  \right)^3}   ~~~, 
\ee
with $\sigma =\tau(\tau-1)$.  In this case, the $r_b$ parameter in transformation (\ref{Kubert-transf}) 
is given by $\tau^3$ times a rational fraction in $\sigma$.
The solution for torsion group $\mathbb{Z}_{9}$ is given in (\ref{eq:Z9a}).
\end{proof}

\begin{proposition}\label{PropositionNa}
The number $n_a$ of primitive solutions to Eq.~(\ref{eq:general3variables}), up to a reordering of the variables or an overall 
sign change,  satisfies  $0 \leq n_a \leq 3$ or is infinite. 
\end{proposition}
\begin{proof}
For each rational point $(X, Y)$ on the elliptic curve (\ref{eq:ellipticcurve}) corresponding  to an $\{ x,y,z \}$ solution to Eq.~(\ref{eq:general3variables}), Proposition \ref{Proposition6sol} ensures that there are only five or two  
other rational points that correspond to the same solution up to a reordering of $ x,y,z $. 
If the rank of the elliptic curve is $r(E) \ge 1$, then the number of rational points is infinite, and thus 
the number of $\{ x,y,z \}$ solutions with gcd$( x,y,z ) =1$ is infinite. 
Up to a reordering or a sign change, there is a single vectorlike solution with gcd$( x,y,z ) =1$, and at most one non-vectorlike solution (see Corollary \ref{Cor2}) with $xyz =0$ for any given $a \neq 1$, 
the number of primitive solutions is $n_a = \infty$ when $r \ge 1$. Examples of rank 1 include $a=-1, 3, 6$.
If the discriminant $\Delta$ given in (\ref{eq:disc}) is 0, then $a = 1$ and $n_a = 0$, or $a= 0,9$ which gives $n_a = \infty$ (the case $a=9$ is solved in Theorem \ref{Proposition-ells}).

It remains to consider the case where  $r(E) = 0$.
In that case, the number of rational points is finite, and is given by the number of elements of the torsion group that are different from $\cal O_\infty$.  
As there are exactly two rational points corresponding to vectorlike solutions, namely $(X,Y) = 3a^2( 1, \pm 3 a (a -1)/2  )$,  
the number of rational points that may correspond to primitive solutions
is 0 for $\mathbb{Z}_3$, 3 for $\mathbb{Z}_6$, 6 for $\mathbb{Z}_9$, and 9 for 
$\mathbb{Z}_{12}$ or $\mathbb{Z}_2 \times \mathbb{Z}_6$. 
Thus, if the torsion group is $\mathbb{Z}_3$ and $r(E) = 0$, then $n_a=0$.  
As an example, for $a=16$ the elliptic curve after  the 
birational transformation $X = 4 \tilde X$, $ Y = 8 \,  \tilde Y$ becomes
$\tilde Y^2 = \tilde X^3 - 432 \tilde X - 16956 $. Using  \texttt{PARI/GP}, we find 
$r(E) = 0$ and the torsion group $\mathbb{Z}_3$ generated by $( \tilde X, \tilde Y ) = ( 48, \pm 270 )$,
corresponding to the vectorlike solution $\{1,-1,0\}$. 

Proposition \ref{Proposition6sol} implies that for $\mathbb{Z}_6$ there is a single primitive solution when $a\neq 4$,
and that solution has two equal variables. An example is $a=-1/11$, and the primitive solution is $\{3,-2,-2\}$. 
For $a = 4$, the torsion group is $\mathbb{Z}_6$ but $n_a$ = 0 because 
$xyz=0$, as shown in the proof to Theorem \ref{theoremnosolnp}. 
From Proposition \ref{Proposition6sol} follows that $n_a = 1$  also occurs for any $a$ (as determined by Theorem \ref{TheoremZ9})
that gives the $\mathbb{Z}_9$ torsion group, 
but there are no two equal variables in this case.
For example, $a=9/73$ gives the primitive solution $\{7,-5, 1\}$.
The $\mathbb{Z}_{12}$ group gives $n_a=2$, and a single primitive solution has two equal variables.
For $a = -1331/8389 $, the primitive solutions are $\{29 , - 20 , - 20 \}$, $\{ 43, -41 , -13 \}$.
The $\mathbb{Z}_2 \times \mathbb{Z}_6$ group gives $n_a=3$, and each solution has two equal variables. 
This is possible only for $\Delta > 0$, {\it i.e.},  $1 < a < 9$. 
For example, $a = 343/127$ gives the primitive solutions $\{ 5,  1, 1\}$,  $\{ 4,  4, -1 \}$, $\{11 , -9 , -9 \}$.
Thus, there are values of $a \in \mathbb{Q}$ for each of the $n_a = 0,1,2,3,\infty$ cases.
The existence of the vectorlike solutions and the Mazur theorem \cite{Mazur} 
ensure that (\ref{eq:ellipticcurve}) cannot have other torsion groups. 
\end{proof}

\section{Integer $a$}
 \label{sec:integera}
\setcounter{equation}{0}

Let us restrict the analysis to $a \in \mathbb{Z}$.  We first identify an infinite family of $a$ values that allow primitive solutions.
We then find that only for $a=9$ are there primitive solutions with two equal variables.  We also present the  general solution for $a=9$, which involves two parameters.  This is followed by a proof that 
the number of primitive solutions for any integer $a$ is either infinite or 0.

\subsection{Infinite families of solutions for integer $a$}
\label{sec:integera-inf}

\begin{proposition}\label{PropositionFib}

For each equation  (\ref{eq:general3variables}) with 
\be
a = (-1)^n  F_n^2  ~~,
\label{eq:Fibonacci}
\ee
where $F_n$ is the Fibonacci number of integer $n\ge 1$, $n\neq 2$, there is at least one primitive solution with $x > -y > -z > 0$.
\end{proposition}
\begin{proof}
\  Using (\ref{eq:arsfamily}) for $s = -4/3$ and $r = p/(9q)$, where $p, q \in \mathbb{Z}$ and  $q\neq 0$, gives the infinite family of 
$a$ values
\be
a = \frac{4 q^2}{ p^2 - 5 q^2}    ~~.
\label{eq:afrac1}
\ee
Proposition \ref{Proposition-a} ensures that Eq.~(\ref{eq:general3variables}) has at least one primitive solution
for each of the $a$ values in (\ref{eq:afrac1}) iff $\pm p/q \neq 3,9$.    
From (\ref{eq:arsfamilyFermat})  follows that a primitive solution is given by
\be
\{ x,y,z \}  = 
\frac { \{ \,   12 q   \;  ,  \, - p \!-\! 9q   \;  ,  \,   p \!-\! 9q   \,  \}   } { 2^{\delta}   \;  {\rm gcd}( p, 3 )    \;  {\rm gcd}( p, q ) } ~~~,
\label{eq:solafrac1}
\ee
where $\delta = 1$ if $pq$ is odd, and $\delta = 0$ otherwise.
An infinite subset of the family of equations (\ref{eq:afrac1}) is obtained for $p = L_n$ and $q= F_n$, where $L_n$ and $F_n$ are the Lucas and Fibonacci numbers with 
$n \ge 1$ ($n = 2$ is not allowed because it gives $p/q = 3$).
The identity $L_n^2 = 5 F_n^2 + 4 (-1)^n$ implies that $a$ is, up to a sign, the square of the  $n$th  Fibonacci number, as in  (\ref{eq:Fibonacci}).
To determine the solution arising from (\ref{eq:solafrac1}) for each $n$,  we use the identities  $L_n = F_n + 2 F_{n-1}$
and  gcd$(F_m,F_n) = F_{\mathrm{gcd}(m,n)}$ (see page 190 of \cite{Hardy}), which imply   $2^{\delta}  {\rm gcd}( L_n,F_n ) = 2$. Furthermore,  $3| L_n$ iff $n\equiv 2\!\! \mod \! 4$, so the primitive 
solution becomes
\be
\{ x,y,z \}  = 
\frac {1} { 3^\theta } \{  \, 6 F_n   \; ,  \, - 5  F_n \!-\!  F_{n-1}   \;  , \,  - 4 F_n  \!+\!  F_{n-1}  \,  \}   ~~,
\label{eq:solafrac1F}
\ee
where $\theta = 1$ for $n\equiv 2\!\! \mod \! 4$, and $\theta = 0$ otherwise. 
This solution to  (\ref{eq:general3variables})  satisfies $x > -y > -z > 0$ for any $n \ge 1$.
\end{proof}

In particular, $n = 1$ gives  $a = -1$ and the solution $\{ 6, -5,  -4 \}$,
while $n = 3$ gives  $a = -4$ and the solution $\{ 12, -11,  -7 \}$.  The case of $n$ even
gives an infinite family of equations with $a$ as a perfect square, discussed further 
in Section \ref{sec:asquare}. Other families of solutions defined by recurrence relations 
can be obtained from the results of \cite{Brueggeman} for an equation related to (\ref{eq:general3tuv}). 

\begin{remark}
For $r = - 6 q^3$, $s = - 6 q^2$, with $q \in \mathbb{Z}_{\ne 0} \,$, Proposition \ref{Proposition-a}
yields the infinite family of negative integers $a = -4 ( 3q^2 - 1)^3$ with the primitive solutions 
$\{ - 6 q^2, \, 1 - 6 q^3, \, 1 + 6 q^3\} $.
For $a = p^3$, Theorem \ref{Theorem1} ensures at least one primitive solution for each $p \in \mathbb{Z}$, $p\neq 1, -2$.
\end{remark}

\subsection{Solutions with $y=z$} 
\label{sec:aan2}

We now show that the Diophantine equation   (\ref{eq:general3variables})  has solutions with two equal variables, when $a$ is integer,
 only for $a =9$ if  $xyz\neq0$.  In the case where one variable is zero,  Corollary \ref{Cor1}   implies that 
the other two variables are equal and nonzero only for $a=4$.

Before proving the main result of this section (Theorem \ref{theorema9}), it is useful to 
solve the following Thue equation with integer variables:
\be  
x^3  + 2 y^3  = 2 ~~.
 \label{eq:xy2}
\ee
This belongs to a family of equations 
shown by Delaunay and Nagell   (see \cite{Haggmark1950, Nagell1969})
to have at most one solution with $xy\neq 0$. We 
prove a stronger statement, using classical methods (the same result can be obtained using a computer algebra system, such as \texttt{PARI/GP}).

\begin{lemma}\label{lemmaxeqy}
The Diophantine equation  (\ref{eq:xy2}) has no solution with $x \neq 0$. 
\end{lemma}
\begin{proof}
If $3|y$, then $x \equiv 2\,$mod$\,3$, implying $x^3 \equiv -1\,$mod$\,9$, in contradiction to $9|y^3$.
If $y \equiv 2\,$mod$\,3$, then $x^3 \equiv 4\,$mod$\,9$, which is impossible for a cube.
For  $y \equiv 1\,$mod$\,3$, Eq.~(\ref{eq:xy2})  uniquely splits into two parts proportional to perfect cubes: \\ [-6mm]
\bear
&& 2 (1 - y) = 9 f^3   \;\;\; , 
\label{eq:3yg1}
\\ [2mm]
&&    \left( y +2 \right)^3  -   \left( y - 1 \right)^3 =  (3g)^3     ~~,
\label{eq:3yg2}
\eear
where  $f,g \in  \mathbb{Z}$, $f,g \neq 0$, \GCD$(f,g) = 1$, and $x = 3fg$. Eq.~(\ref{eq:3yg2}) is in contradiction to 
FLT unless $y=1$ or $y=-2$. These values are not allowed by $x\neq 0$ and Eq.~(\ref{eq:3yg1}).
\end{proof}

\begin{theorem}\label{theorema9}
If there exists a primitive solution $\{x,y,z\}$ to Eq.~(\ref{eq:general3variables}) with $a\in \mathbb{Z}$, $a \neq 9$ and $a\neq 0$, then no two variables can be equal.
For $a = 9$, there are only two primitive solutions with $y = z > 0$: $x = y= 1$ and $x= -5$, $y = 4$. 
\end{theorem}
\begin{proof}
For $y = z$, Eq.~(\ref{eq:general3variables}), takes the form 
\be
a \left(x^3  + 2 y^3   \right) =   (x + 2y  )^3  ~~.
\label{eq:yz}
\ee
For $|a|  \ge  2$, let $p \ge 2$ be a prime divisor of $a$, 
and $n_p \ge 1$ be the multiplicity of $p$ in the prime factorization of $a$.
Eq.~(\ref{eq:yz}) requires $x \equiv - 2y \; {\rm mod} \; p $, which together with \GCD$(x,y) = 1$ implies $p \nmid y$.
If  $3 \nmid n_p$, then Eq.~(\ref{eq:yz}) also requires $x^3  \equiv - 2 y^3 $ mod$\, p$, which gives $p | (6 y^3)$.
Hence, $p = 2$  or   $p = 3$, so that $a$ is of the form $a = 2^{n_2} 3^{n_3} b^3 $, 
with $n_2, n_3, b \in  \mathbb{Z}$, $n_2, n_3 \ge 0$, $b \neq 0$, $2\nmid b$  and $3\nmid b$.      
Eq.~(\ref{eq:yz}) then splits into two equations:
\bear
x + 2y  = b \, c  ~~, 
 \label{eq:xybc1}
\\ [1mm]
2^{n_2} 3^{n_3}( x^3  + 2 y^3 ) = c^3  ~~,
 \label{eq:xybc2}
\eear
where $c  \in  \mathbb{Z}_{\neq 0} $.  Any solution $\{x,y\}$ for $c < 0$ 
implies the solution $\{- x, -y\}$ for $c \to -c$, so it is sufficient to consider  $c > 0$.
The first equation requires  \GCD$(c,y) = $  \GCD$(b,y) = 1$.
Eliminating $x$ from  (\ref{eq:xybc1}) gives 
\be
 2^{n_2+1} 3^{n_3+1}  y (b \, c - y)^2
 = \left( 2^{n_2} 3^{n_3}  \, b^3  - 1 \right) c^3    ~~.
\label{eq:ybc}
\ee
Since \GCD$(c,y(b  c - y) ) = 1$, $c$ cannot have  prime divisors other than 2 or 3.
Given that the parenthesis on the right-hand side of (\ref{eq:ybc})  is not divisible by 2 unless $n_2 = 0$,
and not divisible by 3 unless $ n_3 = 0$,
the above equation can be satisfied only in four cases: 

\noindent
{\it Case 1) }  $n_2 = n_3 = 0$, $c=1$.  
 Eq.~(\ref{eq:xybc2}) becomes $x^3 + 2 y^3  = 1$,
which has no solution with $y\neq 0$ other than $y = -x = 1$  (see \cite{Nagell1969} or page 34 of \cite{Smart1998}).
Thus,  (\ref{eq:xybc1}) implies $b = 1$, so that $a=1$, and Eq.~(\ref{eq:yz}) 
has no primitive solution. 

\noindent
 {\it Case 2) }   $n_2 \equiv 2 \,$mod$\, 3$, $n_3 = 0$,  $c = 2^{(n_2+1)/3}$.   
Eq.~(\ref{eq:xybc2}) is given by $x^3  + 2 y^3  = 2 $, which following Lemma \ref{lemmaxeqy} has no solution with $x\neq 0$. 

\noindent
 {\it Case 3) }  
$n_2 = 0$, $n_3 \equiv 2 \,$mod$\, 3$, $c = 3^{(n_3+1)/3}$. 
Eqs.~(\ref{eq:xybc1})-(\ref{eq:xybc2}) are now $ x + 2y  = 3^{(n_3+1)/3}  \, b$ and $ x^3  + 2 y^3  = 3  $.
The only solutions to the latter equation are $x = y= 1$ and $x= -5$, $y = 4$ \cite{Haggmark1950, 
Nagell1969}. Both solutions correspond to $n_3 = 2$ and $b=1$, so $a = 9$.

\noindent
{\it Case 4) }  $n_2 \equiv 2 \,$mod$\,  3$, $n_3 \equiv 2 \,$mod$\,  3$,   $c = 2^{(n_2+1)/3} 3^{(n_3+1)/3}$.  Eq.~(\ref{eq:xybc2}) becomes  $x^3  + 2 y^3  = 6 $,  and Nagell's theorem  \cite{Haggmark1950, Nagell1969} allows at most one solution.
A solution exists, namely $x = 2$, $y = -1$,  but in this case (\ref{eq:xybc1})  
gives $b  = 0$, so this is not a solution with $a\neq 0$. 
\end{proof}

\subsection{General solution for $a=9$}
\label{sec:a9}

Focusing on Eq.~(\ref{eq:general3variables}) with $a = 9$,
\be
9 \, (x^3  + y^3  + z^3 ) =   (x + y + z)^3  ~~,
\label{eq:a9n3}
\ee
we present a 2-parameter solution, and prove that it is the general solution.  There exist trivial (vectorlike) solutions to 
(\ref{eq:a9n3}) of the type $x = - y$, $z = 0$, and obvious permutations.  All other solutions have $x y z\neq 0$ and are primitive up to 
a  rescaling.  Among these,  Theorem \ref{theorema9} identifies 
$\{1,1,1\}$ and $\{5,-4,-4\}$, which up to an overall sign or a reordering are the only primitive solutions to Eq.~(\ref{eq:general3variables}), for any $a$, with two equal variables.

\begin{lemma}
\label{theorem:generala9}
All integer solutions $\{ x, y, z\}$ of Eq.~{\rm (\ref{eq:a9n3})} with \emph{\GCD}$(x,y,z) = 1$, other than $x=y=z=\pm 1$, 
are given by
\be
\{ x, y, z\} =  \{ \, \left( \ell_1 + \ell_2 \right)   \left(   \ell_1^2 + \ell_2^2  +  \frac{\ell_1 \ell_2 }{2}    \right)    
\; , \; 
 \ell_2^3 - x  
 \; , \; 
 \ell_1^3 - x \,  \}   ~~,
 \label{eq:n3}
\ee
where the integer parameters $\ell_1, \ell_2$ are coprime.
\end{lemma}
\begin{proof}
Changing variables through the inverse of the transformation (\ref{eq:xyztotuv}), which is
\bear
t = x + y \;\; , \;\;\;  u = y + z  \;\; , \;\;\;  v = z + x   ~~,
\label{eq:tuv-xyz}
\eear
we find that Eq.~(\ref{eq:a9n3}) becomes 
\be
 \left( \frac{ t+u+v }{ 3 } \right)^{\! 3} = t \, u \, v  ~~.
 \label{eq:tuvp}
\ee
For primitive solutions, $xyz \neq 0$ so that $tuv \neq 0$, and \GCD$(x,y,z) = 1$ requires $t,u,v$ to be pairwise coprime, unless $|xyz|=1$ in which case $t=u=v=\pm 2$.  We ignore the case of $x=y=z=\pm 1$ from now on.  Since (\ref{eq:tuvp}) implies that $t u v$ is a perfect cube, each of  $t$, $u$, $v$ must be a perfect cube.  Without loss of generality, take $t =\ell_2^3$ and $v=\ell_1^3$ with \GCD$(\ell_1,\ell_2)=1$, which gives $y$ and $z$ in (\ref{eq:n3}). 
 Setting $u=u_0^3$, (\ref{eq:tuvp}) becomes 
\begin{equation}
\left(u_0+\ell_1+\ell_2\right)\left(u_0^2 -(\ell_1+\ell_2)u_0 +\ell_1^2 -\ell_1\ell_2+\ell_2^2 \rule{0mm}{4mm} \right)=0~~.
\label{eq:u0eq}
\end{equation}
The solution $u_0=-\ell_1-\ell_2$ gives the expression for $x$ in  (\ref{eq:n3}). 
There is no other integer solution, because the second factor in (\ref{eq:u0eq}) is quadratic in $u_0$ and has discriminant $-3(\ell_1-\ell_2)^2$.
Note that the solution with zero discriminant,  which  is $t=u=v=\pm 1$, does not correspond to $x,y,z \in\mathbb{Z}$.

For vectorlike solutions $xyz=t\, u \, v=0$.  Any vectorlike solution with \GCD$(x,y,z)=1$ is obtained from (\ref{eq:n3})
when $\ell_1=\pm 1, \ell_2=0$, or $\ell_1=0, \ell_2=\pm 1$, or $\ell_1 =-\ell_2=\pm 1$.
\end{proof}

\begin{corollary}
For any primitive solution to Eq.~{\rm (\ref{eq:a9n3})} other than $\{ 1,1,1\}$, the sum of any two variables is a perfect cube. 
\end{corollary}
\begin{proof}
Theorem~\ref{theorem:generala9} implies that  
\bear
&& t = x+y=\ell_2^3~~, \nonumber\\ [1mm]
\label{eq:xyztotuvtol}
&& u = y+z=-(\ell_1 + \ell_2)^3~~, \\  [1mm]
&& v = z+x =\ell_1^3~~,    
\nonumber  \\ [-0.8cm]  \nonumber 
\eear
where $\ell_1,\ell_2 \in \mathbb{Z}$. 
\end{proof}

\begin{remark}
Using transformation (\ref{eq:anomtoelliptictrans}) for $a = 9$, we see that Eq.~(\ref{eq:a9n3})  becomes a singular curve in $X$ and $Y$. 
Setting $X = 3^4 ( \tilde{X} - 1)$, $Y = 3^6  \tilde{Y}$, the singular curve takes the form
\begin{equation}
 \tilde{Y}^2 =   \tilde{X}^2 \left(  \tilde{X}- 3\right) ~~~.
\end{equation}
All rational solutions to this equation other than $  \tilde{X} =  \tilde{Y} = 0$ (which corresponds to $x=y=z$) 
are given by $  \tilde{X} = \mu^2 + 3$, $  \tilde{Y} = \mu (  \mu^2 + 3)$, with $\mu \in \mathbb{Q}$.
After removing a common factor of $- 3^6$, this corresponds to the solution   
\be
\{  x,y,z \} = 
\{ \, 3 \mu^2 + 5   \; , \;  
 - \left( \mu^3 + 3 \mu + 4 \right)    \; , \; 
  - y - 8  \, \}  ~~.
  \label{eq:xyzmu}
\ee
If $\mu$ is noninteger, this is a rational solution that can be turned into an integer one by an overall rescaling.
Taking $\mu=\left(\ell_1-\ell_2\right)/\left(\ell_1+\ell_2\right)$, and removing an overall factor of 8, reproduces the solution (\ref{eq:n3}).
This provides an alternative proof to Lemma \ref{theorem:generala9}. Note that the vectorlike solution $y+z = 0$ (which corresponds to $\ell_1 = -\ell_2$) 
cannot be recovered for any finite $\mu$. However, the ordering of $x,y,z$ in  (\ref{eq:xyzmu}) is arbitrary, and changing it would allow the $y+z = 0$ solution.
The fact  that not all vectorlike solutions can be obtained for the same ordering is true for any $a$, as can be seen in transformation (\ref{eq:anomtoelliptictrans}).
\end{remark}

\begin{theorem}\label{Proposition-ells}
\ The general solution to Eq.~{\rm (\ref{eq:a9n3})} is  given by 
\bear 
&& x =  \frac{1}{2} \left(  \ell_1^3 + \ell_2^3  + \left( \ell_1 + \ell_2 \right)^3  \rule{0mm}{4mm}  \right)    +  \delta_{\ell_1,0} \, \delta_{\ell_2,0}   ~~,
\nonumber \\ [2mm]
&& y =   \frac{1}{2} \left( - \ell_1^3 + \ell_2^3  -  \left( \ell_1 + \ell_2 \right)^3  \rule{0mm}{4mm}  \right)    +  \delta_{\ell_1,0} \, \delta_{\ell_2,0}  ~~,
\label{eq:general9-deltas}
 \\ [2mm]
&&  z =   \frac{1}{2} \left( \ell_1^3 - \ell_2^3  -  \left( \ell_1 + \ell_2 \right)^3     \rule{0mm}{4mm}  \right)     +  \delta_{\ell_1,0} \, \delta_{\ell_2,0}  ~~,
\nonumber
\eear  
with the two integer parameters satisfying $\ell_1 \ge \ell_2 \ge 1$ and $\ell_1, \ell_2$ coprime, or $\ell_1 = \ell_2 = 0$.  Each 
solution is generated by a unique $\ell_1,\ell_2$ pair.
\end{theorem}

\begin{proof} \
Following Definition \ref{def:general}, we need to show that any primitive solution  to Eq.~(\ref{eq:a9n3})
is obtained for some values of $\ell_1$, $\ell_2$, up to a variable reordering or an overall sign change.
For $\ell_1 = \ell_2 = 0$, (\ref{eq:general9-deltas}) gives the primitive solution $\{  1,1,1  \}$. 
For any $\ell_1 \neq 0$ and  $\ell_2 \neq 0$, it is straightforward to check that (\ref{eq:general9-deltas}) represents 
a solution with $x y z \neq 0$, which means that it is a primitive solution (see Definition \ref{def:primitive}). 

Theorem \ref{theorem:generala9} ensures that any primitive solution other than $\{  1,1,1  \}$ is obtained for 
certain values of $\ell_1$ and $\ell_2$. Primitive solutions have \GCD$(x,y,z) = 1$, and thus can be obtained only 
for \GCD$(\ell_1,\ell_2) = 1$. The reverse is also true, because 
\GCD$(x,y,z) =$  \GCD$(x + y, y + z, z + x) =$ \GCD$(\ell_1 , \ell_2) $.

The interchange of $\ell_1$ and $\ell_2$ leads only to the interchange of $x$ and $y$, and thus the same primitive solution
up to a reordering. Likewise, a sign flip of both $\ell_1$ and $\ell_2$  only flips the sign of the whole set $\{x, y, z\}$. 
Hence, it is sufficient to take $\ell_1 \ge |\ell_2|$. 

Solution (\ref{eq:general9-deltas}) is also invariant under 
\bear
 \ell_2  \,  ,   \ell_1    & \to &     
 - \ell_2  \, , \,  \ell_1 + \ell_2  
~~~,
 \nonumber \\ [-2mm]
 \\ [-3mm]
 x   \,  ,  y    & \to &    - y   \,  ,   -x
 ~~~,
 \nonumber 
\eear
so it is sufficient to take $\ell_2 \ge 1$.
It remains to show that no primitive solution can be obtained for two different values of the $\ell_1,\ell_2$ pair.
To see that, note that the solution (\ref{eq:general9-deltas}) satisfies  (\ref{eq:xyztotuvtol}).
Furthermore, for $\ell_1 \ge \ell_2 \ge 1$ all the solutions satisfy $x \ge |y| \ge |z| \ge 1$,  
so a different $\ell_1,\ell_2$ pair cannot lead to the same solution with a different 
ordering of the variables.
Hence, for each primitive solution $\{ x, y, z\}$  there is a unique choice of $\ell_1,\ell_2$  (the simplest examples are collected in Table  \ref{table:1}). 
\end{proof}

\begin{table}[b!]
\begin{center}
\caption{All primitive solutions to equation (\ref{eq:a9n3}) with $x \le 200$ generated by the general solution (\ref{eq:general9-deltas}), 
and  the corresponding $\ell_1, \ell_2$ parameters. 
 }
 \vspace*{1mm}
\renewcommand{\arraystretch}{1.2}
\begin{tabular}{|c|c|}\hline  
 Primitive solution  $(a=9)$   &   $\ell_1 \, , \ell_2$   
\\ \hline \hline
    $ \{ 1, 1, 1 \} $   &        \hspace*{5mm} $ 0 \, , 0  $   \hspace*{5mm} 
\\   \hline
   \hspace*{5mm}     $  \{ 5, -4, -4  \}      \hspace*{5mm}   $    &   \hspace*{5mm} $  1 \, , 1  $   \hspace*{5mm}   
\\   \hline
   \hspace*{5mm}     $  \{ 18, -17, -10  \}      \hspace*{5mm}   $    &   \hspace*{5mm} $  2 \, , 1  $   \hspace*{5mm}  
\\   \hline
   \hspace*{5mm}     $  \{ 46, -45, -19  \}      \hspace*{5mm}   $    &   \hspace*{5mm} $  3 \, , 1  $   \hspace*{5mm}  
   \\     \hline
   \hspace*{5mm}     $  \{ 80, -72, -53  \}      \hspace*{5mm}   $    &   \hspace*{5mm} $ 3 \, ,  2  $   \hspace*{5mm}  
   \\       \hline
   \hspace*{5mm}     $  \{ 95, -94, -31  \}      \hspace*{5mm}   $    &   \hspace*{5mm} $ 4 \, , 1  $   \hspace*{5mm}  
  \\         \hline
   \hspace*{5mm}     $  \{ 171, - 170, - 46  \}      \hspace*{5mm}   $    &   \hspace*{5mm} $  5 \, , 1  $   \hspace*{5mm}  
\\  \hline 
\end{tabular}  
\label{table:1}
\end{center}
\end{table}

A solution to (\ref{eq:a9n3}) in terms of three integer parameters is given in \cite{Allanach:2019uuu};
since the number of parameters is equal to the number of variables, that does not represent a general solution (see Definition \ref{def:general}).
A solution  to (\ref{eq:tuvp})  in terms of a  rational parameter given in \cite{BremnerGuy} can be shown to be equivalent to 
(\ref{eq:xyztotuvtol}),  but the generality of the solution is not investigated there.

\subsection{Properties of primitive solutions} 
\label{sec:na0orinf}

\begin{theorem}\label{TheoremIntegerNa}
\ For any fixed $a \in \mathbb{Z}$, the number of primitive solutions to 
Eq.~(\ref{eq:general3variables}) is either 0 or $\infty$.
\end{theorem}
\begin{proof}
From Theorem~\ref{Proposition-ells} follows that the number $n_a$ of primitive solutions to Eq.~(\ref{eq:general3variables}) 
for $a=9$ is infinite. As $n_a = 0$ for $a=1$, and $n_a = \infty$ for $a=0$, it is sufficient to consider integer $a$ values 
for which the discriminant (\ref{eq:disc}) is nonzero.
As shown in the proof to Proposition~\ref{PropositionNa},  $n_a$ is finite and nonzero iff 
the elliptic curve (\ref{eq:ellipticcurve}) has rank 0 and the torsion group given by one of the following 
four possibilities:
$\mathbb{Z}_6$, $\mathbb{Z}_9$, $\mathbb{Z}_{12}$ or $\mathbb{Z}_2 \times \mathbb{Z}_6$.
With the exception of $\mathbb{Z}_9$, all these torsion groups imply that at least one 
primitive solution has two equal variables. The latter condition is forbidden by 
Theorem \ref{theorema9}  when $a\in \mathbb{Z}$. 

It thus remains to show that the torsion group cannot be $\mathbb{Z}_9$ for integer $a$.
Theorem \ref{TheoremZ9} establishes that if $\mathbb{Z}_9$ is the torsion group 
of (\ref{eq:ellipticcurve}), then 
\be
a = \frac{q^3}{q^3 - 3 \, p^3}   ~~,
\ee
where $p,q \in \mathbb{Z}$ are coprime, $pq \neq 0$, and
\be
\frac{p}{q} = \frac{2 f ( f - 1) }{f^3 - 3 f^2 + 1} 
\label{eq:pqf}
\ee
with $f \in \mathbb{Q}$, $f \neq 0,1$.
As $a$ is an integer, it follows that $(q^3 - 3 p^3) | q^3$, and 
gcd$(p,q) = 1$ implies that $q^3 - 3 p^3 = 1$, or $q = 3 q_0$ and $9q_0^3 - p^3 = 1$ for $q_0 \in \mathbb{Z}$.
The first of these  two cubic Thue equations has
no solution with $pq \neq 0$. The  $9q_0^3 - p^3 = 1$ equation with $q_0\neq 0$ has only the solution 
 $p=2$, $q_0 = q/3=1$, which is not consistent with the constraint (\ref{eq:pqf}) for any rational $f$.
\end{proof}
Theorem~\ref{TheoremIntegerNa} implies that for any integer $a \neq 0,1,9$ 
the torsion group is $\mathbb{Z}_3$, and its elements are the vectorlike solutions.
Furthermore, there are no primitive solutions iff the elliptic curve (\ref{eq:ellipticcurve}) has rank 0, 
which can be checked using computer algebra systems.
Theorem~\ref{TheoremIntegerNa} also implies that when primitive solutions exist for integer $a$, their number is infinite.
In Table \ref{table:alist10} we show results for all integer $a \neq 0$ with $|a| \leq 10$, indicating 
either a direct proof, or the rank of  (\ref{eq:ellipticcurve}) computed with \texttt{PARI/GP}.
In all the cases other than $a = 9$ shown in Table \ref{table:alist10} where there are primitive solutions (only two of the low-lying ones are displayed), the rank is 1.

\begin{table}[t]
\begin{center}
\caption{Number $n_a$ of primitive solutions to (\ref{eq:general3variables}) for $a \in \mathbb{Z} $,  $|a| \leq 10$.
}
\vspace*{1mm}
\renewcommand{\arraystretch}{1.2}   
\begin{tabular}{|c|cc|c|}\hline  
 {\large $a$} &  {\large $n_a$ }     
 &   Proof   &   Low-lying  primitive solutions 
\\ \hline \hline
1 &   0    & See  (\ref{eq:general3tuv})    &  \\
$-1$ &  $\infty$  & rank 1,  Proposition~\ref{PropositionNa} 
&   $  \{6,-5,-4\} $   ,  \   $  \{1670, -1661, -339\} $       \\ \hline
2 &   0    &    rank 0, Theorem~\ref{TheoremIntegerNa}   &  \\ 
$-2$ &   0   &   Theorem \ref{theoremnosolnp}   &    \\ \hline
3 &  $\infty$  &   rank 1,  Proposition~\ref{PropositionNa}    &   $  \{10,-9,-7\} $   , \   $  \{190,153,-28\} $   \\
$ -3$  &    0     &   Theorem \ref{theromDofs2}   &   \\ \hline
4 & 0 &    Theorem \ref{theoremnosolnp}  &    \\
$ -4$  &  $\infty$   &  rank 1, Proposition~\ref{PropositionNa}    &   $ \{12,-11,-7\} $  ,    $ \{ 23807,-22655,-11640\} $    \\ \hline
5 &    0    &     rank 0, Theorem~\ref{TheoremIntegerNa}  &  \\
$ -5$  &    0    &    Theorem \ref{theromDofs2}        &   \\ \hline
6 & $\infty$  &    rank 1,  Proposition~\ref{PropositionNa}     &   $  \{3,2,1\} $  , \  $ \{20,-17,-15\} $   \\
$ -6$   &    0   &    rank 0, Theorem~\ref{TheoremIntegerNa}  & \\ \hline
7 &   0  &   Theorem \ref{theromDofs2}      &   \\
$ -7$  &   0  &  rank 0, Theorem~\ref{TheoremIntegerNa} &  \\ \hline
8 & $\infty$   &    rank 1,  Proposition~\ref{PropositionNa}     &    $ \{5,4,3\} $   , \  $  \{40,-33,-31\} $    \\
$ -8$  &    0   &  rank 0, Theorem~\ref{TheoremIntegerNa}  &      \\ \hline
9 &  $\infty$     &     \hspace*{-0.2cm}  General solution   (\ref{eq:general9-deltas})  \hspace*{-0.2cm}   &     Table \ref{table:1}   \\
$ -9$  &   0     &  rank 0, Theorem~\ref{TheoremIntegerNa}    &  \\  \hline
$\pm 10$ & 0     &      rank 0, Theorem~\ref{TheoremIntegerNa}    &  \\  \hline
\end{tabular}
\label{table:alist10}
\end{center}
\end{table}

\begin{proposition}\label{propositionMerger}
\  If   $\{ x_i, y_i, z_i \}$, $i = 1,2$, are two different primitive solutions to 
Eq.~(\ref{eq:general3variables}) for any fixed $a \in \mathbb{Q}$,   then  a third solution is 
\be
\{ x_3, y_3, z_3 \} =
\left\{ \, 3   \left( 1 \! - a  \, w^2 \right) \!   - \bar x_1 \! - \bar x_2   \; ,  \; 
 \bar y_2 + w    \left(  x_3  \!  -   \bar x_2   \rule{0mm}{3.8mm} \right)   \; ,  \; 
 a - 1  - y_3   \,
 \rule{0mm}{3.85mm}   \right\}   ~  ,
 \label{eq:Propx3y3z3}
  \ee
where  \\ [-0.6cm]
   \bear
(  \bar x_i \; , \,   \bar y_i  )  
= \frac{  a - 1  }{ y_i + z_i }  \,  ( x_i \; , \, y_i )   \;\;\;  , \;\;\;\; 
 w =  \frac{   \bar y_2    -   \bar y_1   }{  \bar x_2  -   \bar x_1  }    ~~.
 \label{eq:Propx3y3z3defs}
\eear
The  $\{ x_3, y_3, z_3 \}$ solution is not vectorlike iff the two solutions $\{ x_i, y_i, z_i \}$, $i = 1,2$,
do not become identical upon an $x_2 \leftrightarrow y_2$ or  $x_2 \leftrightarrow z_2$ 
transposition and possibly an overall sign change.
If $\{ x_3, y_3, z_3 \} {\cal R} $ with ${\cal R} \in \mathbb{Q}$ is primitive, then 
$\{ x_3, y_3, z_3 \} {\cal R}  \neq  \{ x_i, y_i, z_i \}$ iff  $x_3 \neq \bar x_i$.
\end{proposition}
\begin{proof}
The existence of the primitive solutions $\{ x_i, y_i, z_i \}$, $i = 1,2$, implies $a\neq1$ and $y_i + z_i \neq 0$.
As the solutions are different,  $\bar x_1 \neq \bar x_2$, so that the quantities introduced in (\ref{eq:Propx3y3z3defs})  
do not have singularities.

Let us define the function $\zeta(x,y,z) =  a (x^3 + y^3 + z^3) - (x + y + z )^3$. We find 
\be
\zeta(x_3,y_3,z_3) =\frac{ a -1  }{ \left( \bar x_2 - \bar x_1 \right)^3  } 
\left( P(\bar x_1,\bar x_2, \bar y_1,\bar y_2 )  -  P(\bar x_2,\bar x_1, \bar y_2,\bar y_1 ) \rule{0mm}{3.87mm} \right) ~~.
\ee
The polynomial $P(\bar x_1,\bar x_2, \bar y_1,\bar y_2 ) $ 
is of degree 6, and is given by 
\be    \hspace*{0.1cm}
P(\bar x_1,\bar x_2, \bar y_1,\bar y_2 )  = \bar x_1^6 -  6\bar x_1^5
-3 \bar x_1^4  \left( \bar x_2^2  \!-\!  2\bar x_2   \! +  a  \!-\! 4 \right) +  
\bar x_1^3 \, C(\bar x_2, \bar y_1,\bar y_2)  ,
\ee
where $C(\bar x_2, \bar y_1,\bar y_2) $ is a quadratic polynomial in $\bar x_2, \bar y_1, \bar y_2$ 
with coefficients dependent only on $a$.
Since $\zeta(x_i ,  y_i ,  z_i )=0$ for $i = 1,2$, we obtain the identities
\be
\bar x_i^3 =  3   \bar x_i  \left( \bar x_i + a - 1 \right) - 3a  \bar y_i  \left( \bar y_i - a + 1 \right)  - (a - 1)^3  \;\; , \;\;  i = 1,2 ~~.
\label{eq:xi3}
\ee
We use repeatedly the above identity for $i =1$ to 
eliminate all the powers of $\bar x_1$ higher than 2 in $P(\bar x_1,\bar x_2, \bar y_1,\bar y_2 ) $. 
As a result, $\zeta(x_3,  y_3 ,  z_3 )$ is proportional to $\zeta(x_2 ,  y_2 ,  z_2 )$,
which proves that $\{x_3,y_3,z_3\}$ is a solution to Eq.~(\ref{eq:general3variables}).

The set of two equations $x_3/x_i = y_3/y_i = z_3/z_i$ is solvable iff $x_3 = \bar x_i$, for $i=1,2$. 
Given that the rational rescaling ${\cal R}$ is chosen such that gcd$(x_3 {\cal R}, y_3 {\cal R}, z_3 {\cal R}) = 1$, 
and also gcd$(x_i, y_i, z_i) = 1$,
it follows  that  $x_3 \neq \bar x_i$ is the necessary and sufficient condition for $\{ x_3, y_3, z_3 \} {\cal R} \neq \pm \{ x_i, y_i, z_i \} $.

It remains to determine in which cases is the new solution vectorlike.
The $y_3 , z_3$ variables cannot form a vectorlike pair because $y_3 + z_3 = a - 1 \neq 0$.
Before analyzing the other two pairs, note that $w\neq 0$ because $y_1 \neq y_2$.

If $x_3+ y_3 = 0$, then Eq.~(\ref{eq:general3variables}) gives $z_3 = 0$ because $a\neq 1$.
From (\ref{eq:Propx3y3z3}) then follows that  $y_3 = a -1$, and thus  $x_3 = 1 - a$,
which in turn leads to two constraints.  The first one is $w = - z_2 / ( x_2  + y_2  +  z_2)$, which is equivalent to
\be
\frac{z_2}{z_1} = \frac{x_2  + y_2 }{  x_1  + y_1 }  ~~.
\label{eq:firstcons}
\ee
The second  constraint is that  the expression for  $x_3$ in (\ref{eq:Propx3y3z3})
equals $1 - a$, and can be written as
\be
3 \, a \,  \left(\bar y_2 - a +1  \right)^2  = \left( a + 2 - \bar x_1 - \bar x_2 \right) \left(\bar x_2  + 1 \right)^2 ~~.
\ee
After eliminating $x_1$ from (\ref{eq:firstcons}) and inserting it in the second constraint, we find
\be
\frac{y_1}{z_1} = \frac{x_2}{z_2} +   \frac{ z_2 (-x_2 + y_2+ 2z_2) \, \zeta(x_2,y_2,z_2)  }{ 3 a z_2 (x_2+ y_2) (y_2 + z_2)^2  - z_2 \,\zeta(x_2,y_2,z_2) }  ~~.
\ee
Since $\zeta(x_2,y_2,z_2) =0$, we obtain $y_1/z_1 = x_2/z_2$, and using again (\ref{eq:firstcons})
the result is that $x_3+ y_3 = 0$ iff $x_2/y_1 = y_2/x_1 = z_2/z_1$.  
Taking into account that gcd$(x_i, y_i, z_i) = 1$, the latter condition is equivalent to $x_2 =y_1 $ and $y_2 = x_1$.

Analogously, $x_3+ z_3 = 0$ iff $x_2 =z_1 $ and $z_2 = x_1$. This means that $\{ x_3, y_3, z_3 \} $ is not vectorlike unless
the two initial primitive solutions become identical (possibly up to an overall sign) after one transposition.  
\end{proof}

It is straightforward to extend Proposition~\ref{propositionMerger} to include non-vectorlike solutions that have one variable equal to 0, as given in Corollary \ref{Cor2}. 
Note also that Proposition~\ref{propositionMerger} does not preclude that $\{ x_3, y_3, z_3 \} {\cal R}$ equals to $\pm \{ x_i, y_i, z_i \} $
after a reordering of $x_i, y_i, z_i $. The next corollary takes advantage of the fact that a primitive solution is different than the one obtained by 
any reordering, when no two variables are equal.

\begin{corollary}\label{propositionSelf}
 If   $\{ x_1, y_1, z_1 \}$ is a primitive solution to Eq.~(\ref{eq:general3variables}) for $a \in \mathbb{Z}$,
 then  a non-vectorlike  solution is  
  \be   \hspace*{-0.3cm}
 \{ x, y, z \} = \left\{  3\,   \frac{1 \! - a  w^2 }{a - 1} \!   -   
 \frac{x_1}{ y_1 + z_1 }   \! -   \frac{  y_1 }{ z_1+ x_1  }  
 \; \; ,  \; \;
 w  x     + \frac{ z_1-w \, y_1 }{ z_1 + x_1 }  
 \; \; ,  \; \;
1 - y     \,    \right\}    ~,
  \label{eq:PropSelf}
 \ee
 where  \\ [-0.7cm]
 \be
w =  \frac{  x_1  \, y_1  -  z_1^2 }{ \left(   x_1  -   y_1  \right)  \left(   x_1  +   y_1  +  z_1 \right)     }   ~~.
 \label{eq:defs-self}
 \ee
 For $xyz\neq 0$, this solution becomes primitive upon a rational rescaling, and 
 is different from $\{ x_1, y_1, z_1 \}$  iff  $x \neq  x_1 / (y_1+z_1) $.
\end{corollary}
\begin{proof}
Replacing   $\{ x_2, y_2, z_2 \} $  by $\{ y_1, z_1, x_1 \} $ in (\ref{eq:Propx3y3z3}) and (\ref{eq:Propx3y3z3defs}), 
and removing an overall factor of $a-1$, 
we  obtain (\ref{eq:PropSelf}), so Proposition~\ref{propositionMerger} guarantees that  $\{ x, y, z \}$ is a solution. 
For integer $a$, Theorem \ref{theorema9} ensures that the variables $x_1, y_1, z_1$ are all different,
so that $\{ x_1, y_1, z_1 \} $  differs from $\{ y_1, z_1, x_1 \} $ by two transpositions. 
Proposition~\ref{propositionMerger} then implies that $\{ x, y, z \}$ is not vectorlike. 
Thus, for $xyz\neq 0$ (which is guaranteed by Corollary \ref{Cor1} for $a\neq 3$), 
there exists ${\cal R} \in \mathbb{Q}$ such that $\{ x, y, z \} {\cal R} $ is primitive. 
From Proposition~\ref{propositionMerger} also follows that $\{ x, y, z \} {\cal R}  \neq  \{ x_1, y_1, z_1 \}$ iff  
$x \neq  x_1 / (y_1+z_1) $.
\end{proof}

The fact that two primitive solutions can be combined to produce a third one, as shown in 
Proposition~\ref{propositionMerger}, is related to the group law for addition of rational 
points on elliptic curves. Likewise, the use of a single primitive solution to 
construct a new one, as in (\ref{eq:PropSelf}), is related to the group addition of a rational point to itself. 
For any integer $a$, Theorem \ref{TheoremIntegerNa} ensures that if a primitive solution exists, then the 
number of primitive solutions is infinite. Using a single primitive solution $\{ x_1, y_1, z_1 \}$
to produce (\ref{eq:PropSelf}),  then reordering the new solution
and using it together with  $\{ x_1, y_1, z_1 \}$ to produce 
a third one, and repeating the last step indefinitely would generate an infinite set of primitive solutions.
For example, take $a=-1$ and the solution $\{ 6, -5, -4 \}$. The constructed solution (\ref{eq:PropSelf}),
after a rational rescaling is $\{ 1661, -1670, \, 339 \}$. 
Applying  (\ref{eq:Propx3y3z3}) to $\{ 6, -5, -4 \}$ and $\{ 1661, \, 339, -1670 \}$ gives the primitive  solution
$\{-16490494, \, 17520795, -8126864\}$, and so on.


\subsection{Coefficient $a$ as a perfect square -- applications to  physics}
\label{sec:asquare}

We now turn to Diophantine equations of the type (\ref{eq:general3variables}) 
in the special case where $a \!\in \mathbb{Z}$ is a perfect square.
Putting $a = N^2$, we seek primitive solutions to \\[-7.5mm]
\be
N^2  \left(x^3  + y^3  + z^3  \right) =   (x + y + z )^3  ~~,
\label{eq:asquare}
\ee
where $N \in \mathbb{Z}$, $N \ge 1$.
For $N=1, 2, 4, 5$ there are no primitive solutions, as proved in 
Sections \ref{sec:intro} and  \ref{sec:nosolution}.
For $N= 3$, the general solution to (\ref{eq:asquare}) is given in (\ref{eq:general9-deltas}).

An infinite family of equations, where $a$ is a perfect square, that allow primitive solutions 
is given by $a = k^6$ with $k \in \mathbb{Z}$, $k \neq 0, \pm 1$.
In that case, (\ref{eq:ajk}) implies that a primitive solution to (\ref{eq:asquare}) with $N = k^3$ is
\bear
&&  x =  \left( k^2 + 2  \right)  \left( k^4 +  k^2 + 4  \right)  ~~, 
\nonumber \\ [1mm]
&&  y =  -3  \left(k^4  + 2 k^2 +  3\right)     ~~,  
 \\ [1mm]
&&  z =    - k^6 - 3 k^4  - 6 k^2 + 1    ~~.
 \nonumber
\eear
In particular, $N = 8$ gives the solution $\{ 16, -9,-15\}$, after the \GCD \ is removed. 

Another infinite family is given by the Fibonacci numbers of even index,
\be
N = F_{2k} ~~~,
\label{eq:fib}
\ee
with $k \ge 2$, as follows from Proposition~\ref{PropositionFib}. The solution
\be 
x = 6 F_{2k}    \;\; ,\;\;    y = - 5  F_{2k}  -  F_{2k - 1}    \;\; ,\;\;    z =  - 4 F_{2k}   +  F_{2k-1}   ~~
\label{eq:Fib2k}
\ee
is primitive if $k$ is even, or becomes primitive after dividing by gcd$(x,y,z)=3$ if $k$ is odd.
For $k =2$, we obtain $N = 3$ and the third solution  of Table \ref{table:1}.
For $k =3$, $N = 8$  and the obtained solution is the same as that from the cubic family above.
For $k =4$, $N = 21$, and (\ref{eq:Fib2k}) gives the  solution $\{ 126, -118,-71\}$.

Eq.~(\ref{eq:asquare})  is important for particle physics \cite{Appelquist:2002mw, Cui:2017juz, Allanach:2019uuu}: $x,y,z$ are the charges of three right-handed neutrinos under a new $U(1)$ gauge group, and $N$ is the number of generations of standard fermions, whose $U(1)$ charges are assumed to be generation independent.
Eq.~(\ref{eq:asquare})  represents the only nontrivial combination of anomaly equations for 
a gauge group given by the direct product of the Standard Model $SU(3)\times SU(2)\times U(1)$ group and the new $U(1)$ group, in the presence of three fermions which are singlets under the Standard Model group.

Within the Standard Model,  particle physics experiments have established that the number of fermion generations is  
$N=3$, so that the general solution (\ref{eq:general9-deltas})  for $a = 9$ is particularly relevant. Nevertheless, additional hidden sector 
particles may have another $SU(3)\times SU(2)\times U(1)$ gauge group, and a different number of fermion generations, so that 
the results for an arbitrary integer $N \ge 1$ are still relevant for particle physics.

To understand why Eq.~(\ref{eq:asquare}) must be satisfied by the $U(1)$  charges of three right-handed neutrinos,
one needs to solve first the anomaly equations that arise from triangle diagrams involving two Standard Model gauge bosons and one 
$U(1)$ bosons. This fixes the $U(1)$ charges of all standard fermions in terms of two quark charges, labelled $z_q$ and $z_u$ \cite{Appelquist:2002mw}.
The remaining anomaly equations include one due to a diagram with two gravitons and one $U(1)$ boson,
\be
N \left( z_u - 4 z_q  \right)  = x + y + z  ~~,
\ee
and one due to a diagram with three $U(1)$ bosons,
\be
N \left( z_u  - 4 z_q \right)^3  = x^3 + y^3 + z^3    ~~.
\ee
Eliminating $z_u -4 z_q $ from these equations gives Eq.~(\ref{eq:asquare}).
The primitive solutions are the important ones for particle physics, as they represent the allowed charges for a set of chiral 
({\it i.e.}, non-vectorlike) fermions, which are the only fermions that are likely to be light enough to be 
within experimental reach. 

The family of solutions (\ref{eq:fib}) shows that there are anomaly-free gauge charge assignments
when the number of generations  is given by any Fibonacci number of even index.
Combining with the $N = k^3$ family of solutions, the number of fermion generations 
may belong to the infinite sequence: 3, 8, 21, 27, 55, 64, 125, 144, ...
There are, however,  many values for $N$ different than $k^3$ or $F_{2k}$.
A numerical search utilizing \texttt{PARI/GP} \cite{PARI2} shows that the number of fermion generations may belong to the sequence (truncated at $N \le 150$):
\bear
&& \hspace*{0.2cm} N\in\{3,8,10,17\!\!-\!23,25,27 \!\!-\! 29,32,34 \!\!-\! 39,42,43,47,50 \!\!-\! 53, 55,56,
 \\ && 
 \hspace*{1.1cm} 60, 61, 64 \!\!-\! 66, 69 \!\!-\! 77, 79,81,83,85,86,89, 92 \!\!-\! 95, 97,99, 105,
 \nonumber \\ && 
 \hspace*{1.1cm}
107,109, 111  \!\!-\!  119,122  \!\!-\!  125,127,129, 130, 133  \!\!-\! 135,137,138, 
 \nonumber \\ && 
 \hspace*{1.1cm}
141  \!\!-\!  145,147,148,
\ldots \, \}~~,  
\nonumber
\label{eq:asquareN}
\eear
where the dash between two numbers indicates that all the integers in that range are included in the sequence.
Theorem~\ref{TheoremIntegerNa} implies that the number of primitive solutions to Eq.~(\ref{eq:asquare}) 
is infinite for each value of $N$ in the above sequence.

\vspace*{2mm}

\section{Conclusions}    
\setcounter{equation}{0}\label{sec:conc}


We have solved Diophantine equations of the class ``cube of sum proportional to the sum of cubes''
(\ref{eq:general3variables}), 
for various infinite families of rational $a$ (see Section \ref{sec:infinite-solutions}).
We have proved that  for other infinite families of rational $a$  there exist no primitive solutions (see Section \ref{sec:a-2425}).
The properties of elliptic curves allow the use of numerical methods to decide the solvability for any fixed $a$  (see Section \ref{sec:elliptic}). 
Furthermore, the structure of the torsion group determines the number $n_a$ of primitive solutions.
When the rank of the elliptic curve is 0, $n_a$ can be at most 3 and there are five possible torsion groups, 
while for integer $a$ there are no primitive solutions and the  torsion group is $\mathbb{Z}_3$ (Theorem~\ref{TheoremIntegerNa}). 
Thus, when any primitive solutions exist for some integer $a$ (see, {\it e.g.}, Table \ref{table:alist10}), their number is infinite.
When a primitive solution is known, other ones can be constructed using the method presented in  Corollary \ref{propositionSelf}.
We have also shown that any elliptic curve with torsion group $\mathbb{Z}_{6}$ (for nonzero $j$-invariant),
$\mathbb{Z}_{9}$, $\mathbb{Z}_{12}$, or  $\mathbb{Z}_2 \times \mathbb{Z}_6$,
is a particular case of Eq.~(\ref{eq:general3variables}). 

The case where $a$ is the square of an integer is important for questions that arise in particle physics.
The anomaly equations for certain $U(1)$ gauge extensions of the Standard Model of particle physics lead to 
Eq.~(\ref{eq:general3variables}), where the variables are gauge charges of new fermions,
and the $U(1)$  charges of the quarks and leptons are generation independent.
In particular,  $a=9$ corresponds to  three generations of fermions, as  in the Standard Model, with one right-handed neutrino per generation. 
For $a = 9$ we have presented a  solution,  (\ref{eq:general9-deltas}), 
where the variables $x,\, y,\, z$ are given by homogeneous cubic polynomials in two integer parameters,
and then we have proven that this is the general solution.
More generally, if $a = N^2$, then there are solutions when $N$ is 
given by any Fibonacci number of even index, by any perfect cube (see Section \ref{sec:asquare}),
or by other integers belonging to the sequence shown in (\ref{eq:asquareN}). Seeking an analytic understanding of that sequence remains a challenge.

\vspace*{6mm}

\section*{Appendix: No primitive solutions for $a=4$}    
\setcounter{equation}{0}\label{sec:a4}
\renewcommand{\theequation}{A.\arabic{equation}}

\vspace*{1mm}

{\it This Appendix is based on Section 3.1 of Version v1 of this ArXiv paper; for brevity, this was not included in the version submitted to the journal Communications in Number Theory and Physics.}

\vspace*{2mm}

In this Appendix we prove using classical methods that Eq.~(\ref{eq:general3variables}) with $a = 4$ has no primitive solutions (for an alternative proof, see Theorem \ref{theoremnosolnp}).
A first step is the following Lemma, whose proof is based on infinite descent. 

\vspace*{2mm}

\begin{lemma}\label{lemmaa4}
Let  $ q, r, t \in \mathbb{Z}$ with $|q r t| > 1$. If \GCD$(q, r, t) =1$, then there are no solutions to the equation  
\be
q^4  - r^4 + q^2 r^2  = t^2 ~~.
\label{eq:lemma1}
\ee
\end{lemma}
\begin{proof}
The quadratic equation (\ref{eq:lemma1}) in $q^2$ has the discriminant given by $\Delta = 5 r^4 + 4 t^2$. 
Thus, a necessary condition for the existence of integer solutions to (\ref{eq:lemma1}) is $\Delta = k^2$ where $k \in \mathbb{Z}$ and $k \ge 3$.
This implies that $k$ must satisfy two equations: \\ [-7mm]
\bear
& k  = 2 q^2 + r^2  ~~, 
\label{eq:lemma-q}
\\ [2mm]
& \left( k + 2t \right) \left( k - 2t \right)  = 5 \, r^4  ~~.
\label{eq:lemma-split}
\eear
Since  \GCD$(q, r, t) =1$,  Eq.~(\ref{eq:lemma1})  requires \GCD$(r, t) =1$, so that 
$k + 2t$ and $k - 2t$ are coprime. 
As a result, Eq.~(\ref{eq:lemma-split}) splits into two equations; up to 
a sign flip $t \leftrightarrow -t$, these  can be written as  
\be
k + 2t = 5 r_1^4   \;\;  ,  \;\;  k - 2t = r_2^4  ~~,
\label{eq:lemma-2eqs}
\ee
with $r = r_1 r_2$ and \GCD$(r_1, r_2) =1$.
Since $k + 2t$ and $k - 2t$ have the same parity, both $r_1$ and $r_2$ must be odd.
From Eqs.~(\ref{eq:lemma-q}) and  (\ref{eq:lemma-2eqs}) it follows that 
\be
\left( r_1^2  -  r_2^2 \right)^2  +   4 r_1^4  = 4 q^2  ~~.
\ee
Thus, the three integers $\left| r_1^2  -  r_2^2 \right|$, $2 r_1^2 $, and $2 |q|$ form a Pythagorean triple, which implies 
that there must exist $\ell, m  \in \mathbb{Z}$ such that 
\be
2 |q| = \ell^2 + m^2  ~~,
\ee
and either $ r_1^2  -  r_2^2  = 2 \ell m $, $2 r_1^2  = \ell^2 - m^2$, or  $r_1^2  -  r_2^2 = \ell^2 - m^2  $, $2 r_1^2  =  2 \ell m$.
The first case can be dismissed by noticing that the integer equation $2 r_1^2  = (\ell + m)(\ell - m)$ has no solutions
with $r_1$ odd because $\ell + m$ and $\ell - m$ have the same parity. Let us analyze the second case, which can be written as 
\bear
r_1^2  & \! = \! & \ell m  ~~,
\label{eq:r1lm}
\\ [2mm]
 r_2^2 & \! = \! &  \ell m + m^2  - \ell^2      ~~.
\label{eq:r2lm}
\eear
It is sufficient to consider $ \ell, m \ge 1$. Since \GCD$(r_1, r_2) =1$, the above two equations imply \GCD$(\ell, m) =1$. As a consequence,
Eq.~(\ref{eq:r1lm}) requires both $\ell$ and $m$ to be perfect squares. Let $\ell_0$ and $m_0$ be nonzero integers such that $\ell = \ell_0^2$ and $m= m_0^2$.
From Eq.~(\ref{eq:r2lm}) it follows that
\be
m_0^4 - \ell_0^4 +   m_0^2  \ell_0^2        =  r_2^2    ~~.
\ee
Note that this equation is identical with Eq.~(\ref{eq:lemma1}) upon substituting the variables $m_0, \ell_0, r_2$ with
$q,r,t$, respectively.
If Eq.~(\ref{eq:lemma1}) has any solutions, then there is a minimum value of $r^2 \neq 0$ that satisfies the equation. In particular, that  value 
is the minimum one provided $r^2 \leq \ell_0^2$. At the same time, 
the minimum value of $r^2$ can be written as $r^2 = r_1^2  \, r_2^2 =  \ell_0^2 m_0^2  \, r_2^2  \ge  \ell_0^2$.
These conditions are simultaneously satisfied only if  $r^2 = \ell_0^2$, implying $m_0^2 = \ell_0^2 = 1$, which gives
$|q| = |r| = |t| = 1$. This is in contradiction with $|q r t| > 1$. 
 \end{proof}

\smallskip\smallskip

\begin{theorem}\label{theorema4}
The cubic Diophantine equation 
   \be
4 \left(x^3  + y^3  + z^3  \right) =   (x + y + z )^3  ~~
\label{eq:a4n3}
  \ee
has no primitive solutions.
\end{theorem}
\begin{proof}
If $\{x, y, z\}$ is a primitive solution to Eq.~(\ref{eq:a4n3}), then 
the set $\{x, y, z\}$ is coprime, so at least one variable must be odd.
There is no solution with exactly one or three odd variables because the right-hand side must be a multiple of 4.
The remaining possibility is that there are exactly two odd variables. Without loss of generality, take
$x = 2 s_1$, $y= s_+ +s_-$, $z = s_+ - s_-$,  with $s_1, s_\pm \in \mathbb{Z}$,  $s_+  \ge 1$, $s_1 \neq 0$ and $s_+,s_-$ of opposite parity.
As primitive solutions have $xyz\neq 0$,  Eq.~(\ref{eq:a4n3}) has no rational primitive solutions with $y = z$, so that $s_-\neq 0$. 
From \GCD$(x, y, z) = 1$ follows that \GCD$(s_1, s_+, s_- ) = 1$. 
Using the $s_1, s_\pm$ variables, Eq.~(\ref{eq:a4n3})  becomes
\be
s_1^3 = s_+ \left(   s_1 s_+  + s_1^2   - s_-^2  \right)   ~~.
\label{eq:a4n3s}
\ee

If $s_+  = 1$, then $s_-$ is even and   $s_-^2  =s_1 ( 1+ s_1 - s_1^2 )$.  
The latter equation may have solutions only if $- s_1$ is a perfect square and $s_-^2  = - \ell^2  s_1$  with $\ell \in \mathbb{Z}$.
Thus, $\ell^2 = s_1^2 - s_1 - 1$  which can be written as $ (2s_1-1)^2 - 4\ell^2 = 5$. The only solutions to this equation, consistent with $- s_1$ being a perfect square, are $\ell = \pm 1$, $s_1 =-1$, 
so that $s_-^2 = 1$,  in contradiction with $s_-$ even.

For $s_+  \ge  2$, let $p \ge 2$ be a prime divisor of $s_+$. 
Let $n_p \ge 1$ be  the multiplicity of $p$ in the prime factorization of $s_+ $.
If  $3\nmid n_p$, then Eq.~(\ref{eq:a4n3s}) implies that both $s_1$ and $s_1 s_+ +  s_1^2 - s_-^2$ are multiples of $p$. In that case, $s_-$ is a multiple of $p$, which is in contradiction with \GCD$(s_1, s_+, s_- ) = 1$. 
Thus,  $3|n_p$   for any prime factor, so that  $s_+$ is a perfect cube: 
$s_+ = q_\star^3$ with $q_\star \in \mathbb{Z}$ and $q_\star \ge 2$. From Eq.~(\ref{eq:a4n3s}) follows that $s_1 = r_0 q_\star$  with $r_0 \in \mathbb{Z}$, $r_0 \neq 0$,  and  
\be
s_-^2 =   r_0 (q_\star^4  + q_\star^2 \, r_0 - r_0^2)   ~~.
\label{eq:srq}
\ee

If $r_0   = \pm 1$, then Eq.~(\ref{eq:srq}) becomes $ (2 q_\star^2 \pm 1)^2 \mp 4s_-^2  = 5 $,  
and the only solutions require $q_\star^2 = s_-^2 = 1$.
These solutions, however, are not allowed by the condition that $s_+$ and $s_-$ have opposite parity. 

For $|r_0| \ge 2$, let $p' \ge 2$ be a prime divisor of $r_0$ of multiplicity $n_{p'} \ge 1$.  If $n_{p'} $ were odd, then  Eq.~(\ref{eq:srq})  would imply that both $s_-$ 
and $q_\star^4  + q_\star^2 \, r_0 - r_0^2$ are multiples of $p'$, leading to \GCD$(q_\star, r_0, s_-) \ge p' $, which  is not allowed by  \GCD$(s_1, s_+, s_- ) = 1$. 
Thus, $n_{p'} $  is even, so that $|r_0|$ is a perfect square: $|r_0| = r_\star^2$ with $r_\star \in \mathbb{Z}$ and $r_\star \ge 2$.   
  Eq.~(\ref{eq:srq})  implies   $s_- = t \, r_\star $ with $t \in \mathbb{Z}$, $t \neq 0$, 
  and  now takes the form  
   \be
q_\star^2 \,  r_\star^2  +  \frac{r_0}{|r_0|}  \left( q_\star^4  -  r_\star^4 \right)  =  t^2 ~~.
\label{eq:a4n3oddstar}
 \ee 
Since \GCD$(s_1, s_+, s_- ) = 1$ implies \GCD$(q_\star, r_\star, t) = 1$.
Comparing with the notation of  Eq.~(\ref{eq:lemma1}), the substitutions
$q_\star \to q$ and $r_\star \to r$ for $r_0 > 0$, or $q_\star \to r$ and $r_\star \to q$ for $r_0 < 0$, 
applied to (\ref{eq:a4n3oddstar}) show that Lemma \ref{lemmaa4} forbids any primitive solutions.
\end{proof}

\bigskip

\smallskip


\end{document}